\documentclass[11pt, a4paper, leqno]{article}
\usepackage[margin=2.4cm]{geometry}
\usepackage[utf8]{inputenc} 
\usepackage[T1]{fontenc}
\usepackage[english]{babel}

\usepackage[runin]{abstract}
\usepackage{titling}

\usepackage{amsmath}
\usepackage{amsthm}
\usepackage{amsfonts}
\usepackage{amssymb}
\usepackage{mathtools}
\usepackage{clrscode}
\usepackage{enumitem}
\usepackage{mdwlist}
\usepackage{subscript}

\abslabeldelim{.}

\newcommand{\keywordsname}{Key words}
\makeatletter
\newcommand{\keywords}[1]{%
\def\thekeywords{#1}%
\begin{@bstr@ctlist}
\hspace*{\abstitleskip}{\abstractnamefont\keywordsname\@bslabeldelim}\abstracttextfont\
#1%
\par\end{@bstr@ctlist}
}
\makeatother

\newcommand{\subjclassname}{Mathematics subject classification}
\makeatletter
\newcommand{\subjclass}[2][2010]{%
\begin{@bstr@ctlist}
\hspace*{\abstitleskip}{\abstractnamefont\subjclassname\ (#1)\@bslabeldelim}\abstracttextfont\
#2%
\par\end{@bstr@ctlist}
}
\makeatother

\makeatletter
\def\and{
	\end{tabular}%
	and%
	\begin{tabular}[t]{c}}%
\makeatother

\makeatletter
\def\thanks#1{
\protected@xdef\@thanks{\@thanks
\protect\footnotetext[\the\c@footnote]{#1}}%
}
\makeatother

\makeatletter
\let\addresses\@empty      
\newcommand{\address}[2][]{\g@addto@macro\addresses{\address{#1}{#2}}}
\newcommand{\curraddr}[2][]{\g@addto@macro\addresses{\curraddr{#1}{#2}}}
\newcommand{\email}[2][]{\g@addto@macro\addresses{\email{#1}{#2}}}
\newcommand{\urladdr}[2][]{\g@addto@macro\addresses{\urladdr{#1}{#2}}}
\def\enddoc@text{
  \ifx\@empty\addresses \else\@setaddresses\fi}
\AtEndDocument{\enddoc@text}
\def\emailaddrname{E-mail address}
\def\@setaddresses{\par
  \nobreak \begingroup
  \interlinepenalty\@M
  \def\address##1##2{\begingroup%
    \par\addvspace\bigskipamount
    \@ifnotempty{##1}{(\ignorespaces##1\unskip) }%
    {\noindent\ignorespaces##2}\par\endgroup}%
  \def\email##1##2{\begingroup
    \@ifnotempty{##2}{\nobreak\noindent\emailaddrname
      \@ifnotempty{##1}{, \ignorespaces##1\unskip}\/:\space
      \ttfamily##2\par}\endgroup}%
  \addresses
  \endgroup
}

\makeatother

\makeatletter
\def\cstar#1{\expandafter\@cstar\csname c@#1\endcsname}
\def\@cstar#1{\ifcase#1\or $\ast$\or $\ast\ast$\or $\ast\ast\ast$\fi}
\AddEnumerateCounter{\cstar}{\@cstar}{$\ast\ast\ast$}
\makeatother

\newlist{conditions}{enumerate}{1}
\newlist{nconditions}{enumerate}{1}
\setlist[conditions]{label=\normalfont(\alph*),ref=\normalfont\alph*}
\setlist[nconditions]{label=\normalfont(\arabic*),ref=\normalfont\arabic*}

\mathchardef\mhyphen="2D

\newcommand{\AB}{\mathbb{A}}
\newcommand{\CB}{\mathbb{C}}
\newcommand{\DB}{\mathbb{D}}
\newcommand{\N}{\mathbb{N}}
\newcommand{\PB}{\mathbb{P}}
\newcommand{\R}{\mathbb{R}}

\newcommand{\A}{\mathcal{A}}
\newcommand{\OC}{\mathcal{O}}

\newcommand{\hot}{(\mathrm{hot})}

\newcommand{\codim}{\func{codim}}
\newcommand{\Gr}{\func{Gr}}
\newcommand{\Pole}{\func{Pole}}
\newcommand{\Sing}{\func{Sing}}

\newtheorem{theorem}{Theorem}[section]
\newtheorem{corollary}[theorem]{Corollary}
\newtheorem{proposition}[theorem]{Proposition}
\newtheorem{lemma}[theorem]{Lemma}
\theoremstyle{definition}
\newtheorem*{acknowledgements}{Acknowledgements}
\newtheorem{definition}[theorem]{Definition}
\newtheorem{example}[theorem]{Example}
\newtheorem{remark}[theorem]{Remark}
\newtheorem{case}{\indent Case}

\title{\bf Curve-rational functions}
\date{}
\author{J\'anos Koll\'ar, Wojciech Kucharz \and Krzysztof Kurdyka}

\address{J\'anos Koll\'ar \\
Department of Mathematics, Fine Hall, Washington Road
\\
Princeton University,  Princeton, NJ 08544-1000  USA}
\email {kollar@math.princeton.edu}

\address{Wojciech Kucharz\\Institute of Mathematics, Faculty of Mathematics and Computer
Science,  Jagiellonian University {\L}ojasiewicza 6, 30-348
Krak\'ow, Poland}
\email{Wojciech.Kucharz@im.uj.edu.pl}

\address{Krzysztof Kurdyka\\Laboratoire de Math\'ematiques (LAMA) UMR 5127
CNRS, Universit\'e  Savoie Mont Blanc \\ Campus Scientifique 73 376 Le
Bourget-du-Lac Cedex France}
\email{Krzysztof.Kurdyka@univ-savoie.fr}

\usepackage[pdftex, pdfauthor={\theauthor}, pdftitle={\thetitle}]{hyperref}
\usepackage{chngcntr}
\newcounter{proof}
\counterwithin*{case}{proof}
\counterwithin*{equation}{theorem}

\begin{document}
\maketitle
\thispagestyle{empty}

\begin{abstract}
Let $W$ be a subset  of the set of real points of a real algebraic variety $X$.
We investigate which  functions $f: W \to \R$ are the restrictions of rational functions on $X$.
We introduce two new notions:  {\it curve-rational functions}  (i.e., continuous rational on algebraic curves) and
{\it arc-rational functions} (i.e.,  continuous rational on arcs of algebraic curves). We prove that  under  mild assumptions the following classes of functions coincide: continuous hereditarily rational
(introduced recently by the first named author),
curve-rational and arc-rational. In particular, if $W$ is semialgebraic and $f$  is arc-rational,  then $f$ is  continuous and semialgebraic.
We also show that an  arc-rational function defined  on an open set is arc-analytic (i.e., analytic on analytic arcs).
Furthermore, we study rational functions on products of  varieties. As an application we obtain a characterization of regular functions.
Finally, we get analogous results in the framework of complex algebraic varieties.
\end{abstract}

%

\section{Introduction}\label{sec-1}

In this paper, a real algebraic variety is a quasi-projective variety
$X$ defined over $\R$. We always assume that $X$ is reduced but allow it
to be reducible. By a subvariety we mean a closed subvariety. The set of
real points is denoted by $X(\R)$ and regarded as a topological space
(with the Euclidean topology). It is easy to see that there is an open
affine subset $X^0 \subset X$ that contains $X(\R)$. Thus, as in
\cite{bib4}, one can always view $X(\R)$ as an algebraic subset of
$\R^n$ for some~$n$. In particular, $\AB^n(\R) = \R^n$. 


We are interested in real-valued functions, defined on some subset of
$X(\R)$, that are restrictions of regular functions or rational
functions on $X$. The precise definition is as follows.

\begin{definition}\label{def-1-1}
Let $X$ be a real algebraic variety, and $f \colon W \to \R$ a function
defined on some subset $W \subset X(\R)$.

We say that $f$ is \emph{regular at a point $x \in W$} if there exist a
Zariski open neighborhood $X_x \subset X$ of $x$ and a regular function
$\Phi_x$ on $X_x$ such that $f|_{W \cap X_x} = \Phi_x|_{W \cap X_x}$.
Moreover, $f$ is called a \emph{regular function} if it is regular at
every point in $W$. Thus, regarding $X(\R)$ as an algebraic  subset of $\R^n$, the function
$f$ is regular at $x$ if and only if there exist two polynomials $p,q \in \R[x_1, \dots,x_n]$ such that $q(x) \ne0$ and $f= p/q$ 
on $W\cap \{q\ne 0\}$.

Denoting by $Y$ the Zariski closure of $W$ in $X$, we see that $f$ is
regular at $x$ if and only if $f|_{W \cap Y_x} = F_x|_{W \cap Y_x}$ for
some regular function $F_x$ defined on a Zariski open neighborhood $Y_x
\subset Y$ of $x$.

We say that $f$ is a \emph{rational function} if there exist a Zariski
open dense subset $Y^0 \subset Y$ and a regular function $F$ on $Y^0$
with $f|_{W \cap Y^0} = F|_{W \cap Y^0}$. In other words, $f$ is a
rational function if and only if there exist a rational function $R$ on
$Y$ and a Zariski open dense subset $Y^0 \subset Y$ such that $Y^0
\subset Y \setminus \Pole(R)$ and $f|_{W \cap Y^0} = R|_{W \cap Y^0}$,
where $\Pole(R)$ stands for the polar set of $R$.
\end{definition}

It easily follows that each regular function on $W$ is also a rational
function.

While the definition makes sense for an arbitrary subset $W$, it is
sensible only if $W$ contains a sufficiently large portion of $Y(\R)$.
The key examples of interest are open subsets and semialgebraic subsets,
with $W = X(\R)$ being the most important case.

We are mainly interested in \emph{continuous rational functions} on $W$,
that is, continuous functions which are also rational.

The following are standard examples.

\begin{example}\label{ex-1-2}
The function $f \colon \R^2 \to \R$, defined by
\begin{equation*}
f(x,y) = \frac{x^3}{x^2 + y^2}\quad \textrm{for}\ (x,y) \neq (0,0)\quad
\textrm{and}\quad f(0,0)=0,
\end{equation*}
is continuous rational but it is not regular at $(0,0)$.

The function $g(x,y) = 1/(1 + x^2 + y^2)$ is regular on $\R^2$.
\end{example}

\begin{example}\label{ex-1-3}
Consider the curve $C \coloneqq (x^3 - y^2 = 0) \subset \AB^2$ and the
functions $f$, $g$ defined on $C(\R)$ by
\begin{gather*}
f(x,y) = \frac{y}{x}\quad \textrm{for}\ (x,y) \neq (0,0)\quad
\textrm{and}\quad
f(0,0)=0,\\
g(x,y) = \frac{x}{y}\quad\textrm{for}\ (x,y) \neq (0,0)\quad
\textrm{and}\quad
g(0,0)=0.
\end{gather*}
Then $f$ is continuous rational, whereas $g$ is rational but it is not
continuous at $(0,0)$.
\end{example}

Regular functions on ${W = X(\R)}$, of course, are in common use
\cite{bib4}. On the other hand, continuous rational functions on ${W =
X(\R)}$ have only recently become the object of serious research. Their
algebraic and geometric properties were considered in \cite{bib9, bib10,
bib11, bib20, bib25}. The homotopy and approximation properties of maps
defined by continuous rational functions were studied in \cite{bib12,
bib13, bib14, bib15, bib17, bib27}, and applications of such maps to
algebraic and stratified-algebraic vector bundles were given in
\cite{bib2, bib16, bib18, bib19}.

Several examples discussed in \cite{bib11} show that continuous rational
functions on ${W = X(\R)}$ behave in a rather unusual way. To eliminate
some unexpected and undesirable phenomena, the notion of hereditarily
rational function was introduced in \cite{bib11}. Such functions played
an important role in \cite{bib9, bib16, bib18, bib19, bib25, bib27}.

\begin{definition}\label{def-1-4}
With notation as in Definition~\ref{def-1-1}, $f \colon W \to \R$ is
called a \emph{hereditarily rational function} if for every real
subvariety $Z \subset X$, the restriction $f|_{W \cap Z}$ is a rational
function.
\end{definition}

If $X$ is smooth, then every continuous rational function on $W = X(\R)$
is hereditarily rational \cite[Proposition~8]{bib11}. It is not the case
for singular varieties. We now recall \cite[Example~2]{bib11}.

\begin{example}\label{ex-1-5}
Consider the algebraic surface
\begin{equation*}
S \coloneqq (x^3 - (1 + z^2)y^3 = 0) \subseteq \AB^3.
\end{equation*}
Then ${S(\R) \subset \R^3}$ is an analytic submanifold and the function
${f\colon S(\R) \to \R}$, defined by \linebreak ${f(x, y, z) = (1+z^2)^{1/3}}$, is
analytic and semialgebraic. Furthermore, $f$ is a continuous rational
function on $S(\R)$ since $f(x,y,z) = x/y$ on $S(\R)$ without the
$z$-axis. On the other hand, $f$ restricted to the $z$-axis is not a
rational function. Thus $f$ is not hereditarily rational.
\end{example}

It turns out that hereditarily rational functions can be characterized
by restrictions to irreducible real algebraic curves.

\begin{definition}\label{def-1-6}
With notation as in Definition~\ref{def-1-1}, $f \colon W \to \R$ is
said to be \emph{rational on algebraic curves} if for every irreducible
real algebraic curve $C \subset X$, the function $f|_{W \cap C}$ is
rational. If, in addition, $f|_{W \cap C}$ is continuous, then $f$ is
said to be \emph{continuous rational on algebraic curves} or
\emph{curve-rational} for short.
\end{definition}

Our main result on curve-rational functions is the following.

\begin{theorem}\label{th-1-7}
Let $X$ be a real algebraic variety and let $W \subset X(\R)$ be a
subset that is either open or semialgebraic. For a function $f \colon W
\to \R$, the following conditions are equivalent:
\begin{conditions}
\item\label{th-1-7-a} $f$ is continuous and hereditarily rational.

\item\label{th-1-7-b} $f$ is curve-rational.
\end{conditions}
\end{theorem}

A function on $\R^n$ that is rational on algebraic curves need not be
rational.

\begin{example}\label{ex-1-8}
Consider the transcendental curve
%
$
T \coloneqq (e^x - y = 0) \subset \R^2.
$
%
The function $f \colon \R^2 \to \R$, defined by
\begin{equation*}
f(x,y)=0\quad \textrm{for}\ (x,y) \in T\quad \textrm{and}\quad
f(x,y)=1\quad \textrm{for}\ (x,y) \in \R^2 \setminus T,
\end{equation*}
is rational on algebraic curves but it is not rational.
\end{example}

In Section~\ref{sec-4} we give a detailed description of relationships
between hereditarily rational functions (not necessarily continuous) and
functions rational on algebraic curves.

It is convenient to have the following local variant of
Definition~\ref{def-1-6}.

\begin{definition}\label{def-1-9}
With notation as in Definition~\ref{def-1-1}, $f \colon W \to \R$ is
said to be \emph{continuous rational on algebraic arcs} or
\emph{arc-rational} for short if for every point $x \in W$ and every
irreducible real algebraic curve $C \subset X$, with $x \in C(\R)$,
there exists an open neighborhood $U_x \subset W$ of $x$ such that
$f|_{U_x \cap C}$ is a continuous rational function.
\end{definition}

In Definition~\ref{def-1-9}, one could require only that $f|_{U_x \cap
C}$ be a rational function (not necessarily continuous), but such a
weaker notion would not be useful for us.

Clearly, any curve-rational function is arc-rational. The converse does
not hold for a rather obvious reason. For instance, consider the
hyperbola $H \coloneqq (xy - 1 =0) \subseteq \AB^2$. Any real-valued
function on $H(\R)$ that is constant on each connected component of
$H(\R)$ is arc-rational, but it must be constant to be rational.

Our main result on arc-rational functions concerns functions defined on
connected open sets that avoid singularities.

Let $X$ be a real algebraic variety. We say that an open subset $U
\subset X(\R)$ is \emph{smooth} if it is contained in $X \setminus
\Sing(X)$, where $\Sing(X)$ stands for the singular locus of $X$.

\begin{theorem}\label{th-1-10}
Let $X$ be a real algebraic variety and let $U \subset X(\R)$ be a
connected smooth open subset. For a function $f \colon U \to \R$, the
following conditions are equivalent:
\begin{conditions}
\item\label{th-1-10-a} $f$ is continuous and hereditarily rational.

\item\label{th-1-10-b} $f$ is arc-rational.
\end{conditions}
\end{theorem}
The main properties of arc-rational functions on semialgebraic   sets 
can be summarized as follows.

\begin{theorem}\label{th-1-11a}
Let $X$ be a real algebraic variety and let
$f:W\to \R$ be  an arc-rational function defined on a semialgebraic subset $W\subset X(\R)$.
Then $f $ is continuous and there exists a sequence of semialgebraic sets
\begin{equation*}
W = W_0 \supset W_1 \supset \ldots \supset W_m = \emptyset
\end{equation*}
which are closed in $W$,  such that $f$ is a regular function on each connected component
of $W_i\setminus W_{i+1}$, for $i=0, \dots, m-1$. In particular, $f$ is a  semialgebraic function.
\end{theorem}

We also establish a connection between arc-rational functions and,
introduced earlier in \cite{bib22}, arc-analytic functions.
A~function $\varphi \colon V \to \R$, defined on a real analytic variety
$V$, is said to be \emph{arc-analytic} if $\varphi \circ \eta$ is
analytic for every analytic arc $\eta \colon (-1,1) \to V$. An
arc-analytic function on $\R^n$ need not be continuous \cite{bib1} and even for $n=2$   
it may have a nondiscrete singular set \cite{bib23}.

\begin{theorem}\label{th-1-11}

Let $X$ be a real algebraic variety and let
$f:W\to \R$ be an arc-rational function defined on an open subset $W\subset X(\R)$.
Then $f$ is continuous and arc-analytic.
\end{theorem}

The paper is organized as follows.

In Section~\ref{sec-2}, imposing a weaker condition than in
Definition~\ref{def-1-9}, we introduce \emph{functions regular on smooth
algebraic arcs}. The key result is Theorem~\ref{th-2-4}. It asserts that
a function regular on smooth algebraic arcs, defined on a connected
smooth open set, is rational. 


Section~\ref{sec-4} contains several results that
can be derived from Theorem~\ref{th-2-4}. According to
Theorems~\ref{th-4-3} and~\ref{th-4-5}, a function defined on an open or
semialgebraic set is hereditarily rational, provided that it is rational
on algebraic curves and regular on smooth algebraic arcs.
Theorem~\ref{th-4-7} says that a function defined on a semialgebraic set
is hereditarily rational if and only if it is rational on algebraic
curves and semialgebraic. By Corollary~\ref{cor-4-6}, a function defined
on a semialgebraic set and regular on smooth algebraic arcs is
semialgebraic. This latter fact is important for the proofs of
Theorems~\ref{th-1-7},~\ref{th-1-10}, ~\ref{th-1-11a}  and~\ref{th-1-11} given in
Section~\ref{sec-5}. Rational functions regular on smooth algebraic arcs
need not be continuous (Example~\ref{ex-2-3}) and for this reason we
introduced arc-rational functions.


In Section~\ref{sec-7} we investigate rational functions on products of
varieties. Theorem~\ref{th-7-1} is a substantial generalization of
Theorem~\ref{th-2-4}. It is one of our main results, along with the theorems announced in this section.

Section~\ref{sec-8} is devoted to regular functions.
Theorem~\ref{th-8-1} says that a function defined on a connected open
subset $U \subset \R^n$, with $n \geq 2$, is regular if and only if its
restriction to $U \cap M$ is regular for every $2$-dimensional affine
plane $M \subset \R^n$. A~variant of this result for functions defined
on $X(\R)$, where $X$ is a smooth real algebraic variety, is given in
Theorem~\ref{th-8-2}.

In Section~\ref{sec-9} we consider analogous notions in the framework of
complex algebraic varieties and obtain counterparts of the results
described above.

\section{Functions regular on smooth algebraic arcs}\label{sec-2}

\subsection{The key result}
Let $C$ be an irreducible real algebraic curve. We call any noncompact connected
smooth open subset ${A \subset C(\R)}$ a \emph{smooth algebraic arc}.
Thus, a subset $A \subset C(\R)$ is a smooth algebraic arc if and only
if it is homeomorphic to $\R$ and contained in $C \setminus \Sing(C)$.
If, in addition, $C$ is a curve in a real algebraic variety $X$, we say
that $A$ is a smooth algebraic arc in $X(\R)$.

\begin{definition}\label{def-2-1}
Let $X$ be a real algebraic variety, and $f \colon W \to \R$ a function
defined on some subset $W \subset X(\R)$.
We say that $f$ is \emph{regular on smooth algebraic arcs} if for every
point $x \in W$ and every smooth algebraic arc $A$ in $X(\R)$, with $x
\in A$, there exists an open neighborhood $U_x \subset W$ of $x$ such
that the function $f|_{U_x \cap A}$ is regular (equivalently, one can
require that the function $f|_{U_x \cap A}$ be continuous rational).

Assuming that $W$ is an open subset, we see that $f$ is regular on
smooth algebraic arcs if and only if the restriction of $f$ is a regular
function on each smooth algebraic arc contained in $W$.
\end{definition}

The following example is given just to illustrate the definition.

\begin{example}\label{ex-2-2}
Let $W = \{ x \in \R \mid x = 0\ \textrm{or}\ x \geq 1 \}$ and let $f
\colon W \to \R$ be defined by $f(0) = 0$ and $f(x) = 1/x$ for $x \geq
1$. Then $f$ is regular on smooth algebraic arcs. Clearly, $f$ cannot be
extended to a regular function on $\R$.
\end{example}

Any arc-rational function is regular on smooth algebraic arcs.
A~rational function on $\R^n$ can be regular on smooth algebraic arcs
without being even locally bounded on algebraic curves.

\begin{example}\label{ex-2-3}
\renewcommand{\theequation}{\roman{equation}}
The rational function $f \colon \R^2 \to \R$, defined by
\begin{equation*}
f(x,y) = \frac{x^8 + y(x^2 - y^3)^2}{x^{10} + (x^2 - y^3)^2}\quad
\textrm{for}\ (x,y) \neq (0,0)\quad \textrm{and}\quad f(0,0)=0,
\end{equation*}
has the following properties:
\begin{nconditions}
\item\label{ex-2-3-1} $f$ is not locally bounded on the curve $x^2 - y^3
=0$;

\item\label{ex-2-3-2} $f$ is not arc-rational;

\item\label{ex-2-3-3} $f$ is regular on smooth algebraic arcs.
\end{nconditions}
Conditions (\ref{ex-2-3-1}) and (\ref{ex-2-3-2}) hold since $f(x,y) = 1/x^2$
on the curve $x^2 - y^3 = 0$ away from $(0,0)$. In order to prove
(\ref{ex-2-3-3}), it suffices to show that for any smooth algebraic arc $A
\subset \R^2$, with $(0,0) \in A$, the function $f|_A$ is regular at
$(0,0)$. Such an arc $A$ has near $(0,0)$ a local analytic
parametrization of the form
\begin{align}\label{eq-i}
x(t) &= at + \hot,\quad a\in \R,\\
y(t) &= t + \hot \notag
\end{align}
or
\begin{align}\label{eq-ii}
x(t) &= t + \hot,\\
y(t) &= bt^k + \hot,\quad b \in \R,\ k>1, \notag
\end{align}
where $\hot = \textrm{higher order terms}$. In case \eqref{eq-i},
$f(x(t), y(t)) = t + \hot$ no matter whether $a = 0$ or $a \neq 0$. In
case \eqref{eq-ii},  
\[
  f((x(t),y(t))=\begin{cases}
               bt^k +\hot \hskip1,2 cm \text {if} \hskip0,5 cm  2\leq k\leq 3,\\
            (b+1)t^4+\hot \hskip0,3 cm \text {if} \hskip0,5 cm k=4,\\
            t^4+\hot \hskip1,4 cm  \text {if} \hskip0,5 cm k\ge 5.
            \end{cases}
\]
Thus, $f|_A$ is
regular at $(0,0)$ as required.
%
\renewcommand{\theequation}{\arabic{equation}}
\end{example}

The following result will play a key role in the subsequent sections.

\begin{theorem}\label{th-2-4}
Let $X$ be a real algebraic variety, $U \subset X(\R)$ a connected
smooth open subset, and $f \colon U \to \R$ a function regular on smooth
algebraic arcs. Then there exists a rational function $R$ on $X$ such
that $P \coloneqq U \cap \Pole(R)$ has codimension at least $2$ and
$f|_{U \setminus P} = R|_{U \setminus P}$.
\end{theorem}

\subsection{ Semialgebraic case}
First  we show that Theorem~\ref{th-2-4} holds if $f$ is
assumed to be a semialgebraic function. 

In the proof of the next result, we use Bertini's theorem
\cite[Theorem~3.3.1]{bib24} to produce irreducible real algebraic
curves.

Given integers $1 \leq k \leq N$, we denote by $\Gr(k,N)$ the
Grassmann variety of $k$-dimensional linear subspaces of $\PB^N$.

\begin{proposition}\label{prop-2-5}
Let $X$ be a real algebraic variety, $U \subset X(\R)$ a nonempty smooth
open subset, and $f \colon U \to \R$ a function regular on smooth
algebraic arcs. Assume that  the function~$f$ is 
semialgebraic (so $U$ is a semialgebraic set). Then there exist a nonempty open subset $U_0 \subset U$
and a rational function $R$ on $X$ such that $U_0 \subset X \setminus
\Pole(R)$ and $f|_{U_0} = R|_{U_0}$.
\end{proposition}

\begin{proof}
By replacing $U$ with a smaller subset, we may assume that $X$ is
irreducible. The assertion holds if $\dim X \leq 1$, so suppose that $d
\coloneqq \dim X \geq 2$.

By definition of semialgebraic, there exist a nonempty semialgebraic
open subset $W \subset U$ and an irreducible hypersurface $Y \subset X
\times \AB^1$ such that the graph of $f|_W$ is contained in $Y$. Then
$f|_W$ is a rational function if and only if the first projection $\pi_1
\colon Y \to X$ is birational. Clearly, once we know that $f|_W$ is a
rational function, we immediately obtain $U_0$ and $R$ with the required
properties.

Suppose that $\pi_1 \colon Y \to X$ has degree $m>1$. Fix an embedding $X
\subset \PB^N$. By Bertini's theorem, the set 
$G^* \subset \Gr(N-d-1,N)$ 
consisting of those linear subspaces $L$ for which
$\pi_1^{-1}(X \cap L)$ is $1$-dimensional, irreducible and $\pi_1^{-1}(X
\cap L) \to X \cap L$ (the restriction of $\pi_1$) has degree $m$ is
open and dense in the Zariski topology. Thus there exists an $L \in G^*$
such that $W \cap L$ contains a smooth algebraic arc~$A$. By
construction, the graph of $f|_A$ lies on the irreducible real algebraic
curve $\pi_1^{-1}(X \cap L)$, hence $f|_A$ is not regular, a
contradiction.
\end{proof}

It is not hard to extend a rational representation from an open set to a
larger one.

\begin{lemma}\label{lem-2-6}
Let $X$ be a real algebraic variety, $U \subset X(\R)$ a connected
smooth open subset, and $f \colon U \to \R$ a function regular on smooth
algebraic arcs. Assume that there exists a nonempty open subset $U_0
\subset U$ and a rational function $R$ on $X$ such that $U_0 \subset X
\setminus \Pole(R)$ and $f|_{U_0} = R|_{U_0}$. Then $P \coloneqq U \cap
\Pole(R)$ has codimension at least $2$ and $f|_{U \setminus P} = R|_{U
\setminus P}$.
\end{lemma}

\begin{proof}
If $\dim X \leq 1$, then $f$ is a regular function, hence the assertion
holds. Suppose that $d \coloneqq \dim X \geq 2$. The Zariski closure of
$U$ in $X$ is an irreducible component of $X$, so we may assume that $X$
is irreducible.

First we prove that
\begin{equation}\label{eq-2-6-1}
f|_{U \setminus P} = R|_{U \setminus P}.
\end{equation}
Let $\A$ be the set of all smooth algebraic arcs in $X(\R)$ that are
contained in $U$. We claim that each point $p \in U$ has an arbitrarily
small open neighborhood $U(p) \subset U$ such that any two points of
$U(p)$ belong to an arc in $\A$, contained in $U(p)$. Such a
neighborhood $U(p)$ can be constructed as follows. We can find a Zariski
open neighborhood $X(p) \subset X$ of $p$, a real morphism $\varphi
\colon X(p) \to \AB^d$ and an open neighborhood $V(p) \subset U$ of $p$
such that $\varphi(V(p)) = (-1,1)^d \subset \R^d$, $\varphi(p) = 0$ and
the restriction $\psi \colon V(p) \to (-1,1)^d$ of $\varphi$ is a real
analytic diffeomorphism. If $0 < \varepsilon <1$ and $I \subset
(-\varepsilon, \varepsilon)^d$ is an open interval, then $\psi^{-1}(I)
\subset \psi^{-1}((-\varepsilon, \varepsilon)^d)$ is a smooth algebraic arc.
We can take $U(p) \coloneqq \psi^{-1}((-\varepsilon, \varepsilon)^d)$
for $0 < \varepsilon \ll 1$.

Fix a point $p_0 \in U_0$ and let $p \in U \setminus P$ be an arbitrary
point. Let $\gamma \colon [0,1] \to U$ be a continuous path with
$\gamma(0) = p_0$ and $\gamma(1) = p$. We can cover the compact set
$\gamma([0,1])$ by a finite collection of open sets $U(p_0), U(p_1),
\ldots, U(p_r)$ such that $U(p_0) \subset U_0$, $p_r = p$, and the
intersection $U(p_i) \cap U(p_{i+1})$ is nonempty for all
$i=0,\ldots,r-1$. Now we use induction on $i$ to show that
\begin{equation}\label{eq-2-6-2}
f|_{U(p_i) \setminus P} = R|_{U(p_i) \setminus P}
\end{equation}
for $i = 0, \ldots, r$. This is clear for $i=0$. Suppose that
\eqref{eq-2-6-2} holds for $i=j$, where $0 \leq j < r$. Fix a point $x_0
\in ( U(p_j) \cap U(p_{j+1}) ) \setminus P$ and let $x \in U(p_{j+1})
\setminus P$ be an arbitrary point. We choose an arc~$A$ in $\A$ such
that $A \subset U(p_{j+1})$ and $x_0, x \in A$. The functions $f|_{A
\setminus P}$, $R|_{A \setminus P}$ are regular and equal on the
nonempty open subset $U(p_j) \cap (A \setminus P)$ of $A$, hence $f|_{A
\setminus P} = R|_{A \setminus P}$ and $f(x) = R(x)$. This completes the
inductive proof of \eqref{eq-2-6-2}. Equality \eqref{eq-2-6-1} follows.

It remains to prove that $\codim P \geq 2$. Suppose to the contrary that
$\codim P = 1$. Let $B$ be an arc in $\A$ that meets $P$ transversally
at a general point. Then $f|_B$ is a regular function satisfying
\begin{equation*}
(f|_B)|_{B \setminus P} = f|_{B \setminus P} = R|_{B \setminus P},
\end{equation*}
which means that $R$ cannot have a pole along $B$, a contradiction.
\end{proof}

\subsection {Reduction to the semialgebraic case}
Our goal  now is  to reduce Theorem \ref{th-2-4} to the already known semialgebraic case.
The problem is rather subtle as illustrated by the following.
\begin{example}\label{ex-2-7}
We construct  a continuous arc-semialgebraic function $k:\R^2 \to \R$ which is not semialgebraic.
Let $\gamma >0$ be an irrational number, and set 
$$
\Gamma := \{ (x,y) \in \R^2: \,  x^\gamma -e^{-1/x} < y < x^\gamma
+e^{-1/x}, \, x \in (0,1)\}.
$$
For any semialgebraic curve $C$  there exists $\varepsilon >0$  such that
\begin{equation}\label{eq-2-7}
C\cap\Gamma\cap ((0,\varepsilon) \times \R) =\emptyset.
\end{equation}
This  follows from the Puiseux expansion of the branches (contained in
$\{x>0\} $) of $C$  at the origin.
Indeed, each such branch is of the form $y= h(x^{1/q})$, $x>0$, where $q$ is a positive integer
 and $h: (-\delta,\delta) \to \R$ is an analytic function.
  Set  $c_n = (1/n,(1/n)^\gamma)$ and choose a sequence
$r_n\searrow 0$ such that each ball $B_n: =B(c_n,r_n)\subset \Gamma$
and all these balls are disjoint.
For  $z\in B_n$ we define 
$$
k(z) = \max \{0,r_n- |z-c_n|\}
$$
and put $k\equiv 0$ on the complement of $\bigcup_{n\in \N} B_n$.  Clearly
$k:\R^2\to \R$ is continuous but not semialgebraic.
Actually, $k$ is semialgebraic on any compact semialgebraic set  $K\subset
\R^2\setminus \{(0,0)\}$.

Let $\varphi :(-1,1) \to  \R^2$ be a continuous semialgebraic  arc. Then
by \eqref{eq-2-7} its image meets only finitely many balls $B_n$, hence 
$k\circ \varphi $ is a  semialgebraic (and continuous) function.

\end{example}

For the reduction step we need a result on real analytic functions due to Siciak \cite{bib26}
and B\l{}ocki \cite{bib3}. The following theorem is a special case of
\cite[Theorem~A]{bib3}.

\begin{theorem}\label{th-3-1}
Let $f \colon U \to \R$ be a function defined on a nonempty open subset
$U \subset \R^n$. Assume that the restriction of $f$ is analytic on any
open interval contained in $U$ and parallel to one of the coordinate
axes. Then there exists a nonempty open subset $U_0 \subset U$ such that
the function $f|_{U_0}$ is analytic. \qed
\end{theorem}

Let $X$ be a real algebraic variety, and $U \subset X(\R)$ a
semialgebraic smooth open subset. Recall that a function $f \colon U \to
\R$ is called a~\emph{Nash function} (or an \emph{algebraic function})
if it is analytic and semialgebraic \cite{bib4}.

The following result is contained in \cite[p.~202, Theorem~6]{bib8}.

\begin{theorem}\label{th-3-2}
Let $U = (a_1, b_1) \times \cdots \times (a_d, b_d) \subset \R^d$ be the
product of open intervals and let $f \colon U \to \R$ be an analytic
function. Assume that the restriction of $f$ is a Nash function on any
open interval contained in $U$ and parallel to one of the coordinate
axes. Then $f$ is a  Nash function. \qed
\end{theorem}

We only need the following straightforward consequence of
Theorems~\ref{th-3-1} and~\ref{th-3-2}.

\begin{corollary}\label{cor-3-3}
Let $f \colon U \to \R$ be a function defined on a nonempty open subset
$U \subset \R^n$. Assume that the restriction of $f$ is a Nash function
on any open interval contained in $U$ and parallel to one of the
coordinate axes. Then there exists a nonempty semialgebraic open subset
$U_0 \subset U$ such that the restriction $f|_{U_0}$ is a Nash function.
\qed
\end{corollary}

After these preparations, we can prove a variant of
Corollary~\ref{cor-3-3} in the framework of real algebraic varieties.

\begin{lemma}\label{lem-3-4}
Let $X$ be a real algebraic variety, $U \subset X(\R)$ a nonempty smooth
open subset, and $f \colon U \to \R$ a function regular on smooth
algebraic arcs. Then there exists a nonempty semialgebraic open subset
$U_0 \subset U$ such that the restriction $f|_{U_0}$ is a Nash function.
\end{lemma}

\begin{proof}
By replacing $U$ with a smaller subset, we may assume that $U$ is
semialgebraic and $X$ is irreducible. Setting $d \coloneqq \dim X$ and
shrinking $U$ further if necessary, we can find a nonempty Zariski open
subset $X^0 \subset X$, a real morphism $\varphi \colon X^0 \to \AB^d$
and a semialgebraic open subset $V \subset \R^d$ such that $U \subset
X^0$, $\varphi(U) = V$ and the restriction $\psi \colon U \to V$ of
$\varphi$ is a Nash isomorphism. For any open interval $I \subset V$,
the inverse image $A \coloneqq \psi^{-1}(I)$ is a smooth algebraic arc in
$X(\R)$, hence the restriction $f|_A$ is a regular function. It follows
that $(f \circ \psi^{-1})|_I$ is a Nash function. By
Corollary~\ref{cor-3-3}, there exists a nonempty semialgebraic open
subset $V_0 \subset V$ such that $(f \circ \psi^{-1})|_{V_0}$ is a Nash
function. Thus, $U_0 \coloneqq \psi^{-1}(V_0)$ is a nonempty
semialgebraic open subset of $U$ and the restriction $f|_{U_0}$ is a
Nash function, as required.
\end{proof}

\begin{proof}[Proof of Theorem~\ref{th-2-4}]
It suffices to combine Lemma~\ref{lem-3-4}, Proposition~\ref{prop-2-5}
and Lemma~\ref{lem-2-6}.
\end{proof}

\section{Consequences of Theorem~\ref{th-2-4}}\label{sec-4}
\subsection{Hereditarily
rational functions}
First we record a straightforward characterization of hereditarily
rational functions.

\begin{proposition}\label{prop-4-1}
Let $X$ be a real algebraic variety, and $W \subset X(\R)$ some subset.
For a function $f \colon W \to \R$, the following conditions are
equivalent:
\begin{conditions}
\item\label{prop-4-1-a} $f$ is hereditarily rational.
\item\label{prop-4-1-b} There exists a sequence of sets
\begin{equation*}
W = W_0 \supset W_1 \supset \ldots \supset W_m = \varnothing
\end{equation*}
such that, if $Y_i$ is the Zariski closure of $W_i$ in $X$, then $W_i =
W \cap Y_i$, $Y_i \setminus Y_{i+1}$ is Zariski dense in $Y_i$, $W_i
\setminus W_{i+1} \subset Y_i \setminus \Sing(Y_i)$ and $f|_{W_i \setminus
W_{i+1}}$ is a regular function for $i = 0, \ldots, m-1$.

\item\label{prop-4-1-c} There exists a sequence of sets
\begin{equation*}
W = W_0 \supset W_1 \supset \ldots \supset W_m = \varnothing
\end{equation*}
such that, if $Y_i$ is the Zariski closure of $W_i$ in $X$, then $W_i =
W \cap Y_i$ and $f|_{W_i \setminus W_{i+1}}$ is a regular function for
$i = 0, \ldots, m-1$.
\end{conditions}
\end{proposition}

\begin{proof}
To prove (\ref{prop-4-1-a})~$\Rightarrow$~(\ref{prop-4-1-b}), suppose
that (\ref{prop-4-1-a}) holds, set $W_0 \coloneqq W$ and denote by $Y_0$
the Zariski closure of $W_0$ in $X$. Since $f$ is a rational function,
we can find a real subvariety $Z_1 \subset Y_0$ such that $Y_0 \setminus
Z_1$ is Zariski dense in $Y_0$, $Y_0 \setminus Z_1 \subset Y_0 \setminus
\Sing(Y_0)$ and $f|_{W_0 \setminus Z_1}$ is the restriction of a regular
function on $Y_0 \setminus Z_1$. Set $W_1 \coloneqq W \cap Z_1$ and let
$Y_1$ be the Zariski closure of $W_1$ in $X$. Then $Y_1 \subset Z_1$,
$W_1 = W \cap Y_1$ and $f|_{W_0 \setminus W_1}$ is a regular function.
Note that $\dim Y_0 > \dim Y_1$. Since $f|_{W_1}$ is a rational
function, we can repeat this construction to get $W_2$, and so on. The
process terminates after finitely many steps with $W_m = \varnothing$,
which proves (\ref{prop-4-1-b}).

It is clear that (\ref{prop-4-1-b})~$\Rightarrow$~(\ref{prop-4-1-c}).

Suppose that (\ref{prop-4-1-c}) holds. Let $Z \subset X$ be a real
subvariety, $S$ the Zariski closure of $W \cap Z$ in~$X$, and $T$ an
irreducible component of $S$. We have $T \subset Y_i$ and $T^0 \coloneqq
T \setminus Y_{i+1} \neq \varnothing$ for some $i$. Clearly, $T^0$ is
Zariski open dense in $T$. Furthermore, $W \cap T^0 \subset W_i
\setminus W_{i+1}$, hence $f|_{W \cap T^0}$ is a regular function. It
follows that $f|_{W \cap Z}$ is a rational function. Thus,
(\ref{prop-4-1-c}) implies (\ref{prop-4-1-a}).
\end{proof}
\subsection{Functions defined on open subsets}

To derive from Theorem~\ref{th-2-4} some global results, we have to deal
with functions defined on smooth open sets that are not necessarily
connected.

\begin{lemma}\label{lem-4-2}
Let $X$ be a real algebraic variety, $U \subset X(\R)$ a smooth open
subset, and $f \colon U \to \R$ a function rational on algebraic curves.
Let $\{U_i\}$ be the family of all connected components of~$U$. Assume
that the restrictions $f|_{U_i}$ are rational functions. Then the
following hold:
\begin{nconditions}
\item\label{lem-4-2-1} There exist a rational function $R$ on $X$ and a
family $\{X_i^0\}$ of Zariski open dense subsets of $X$ such that $X_i^0
\subset X \setminus \Pole(R)$ and $f|_{U_i \cap X_i^0} = R|_{U_i \cap
X_i^0}$ for all $i$.

\item\label{lem-4-2-2} $f$ is a rational function if the family
$\{U_i\}$ is finite.

\item\label{lem-4-2-3} $f$ is a rational function if it is regular on
smooth algebraic arcs.
\end{nconditions}
\end{lemma}

\begin{proof}
\setcounter{equation}{\value{nconditionsi}}
It is clear that (\ref{lem-4-2-1}) implies (\ref{lem-4-2-2}).
Furthermore, according to Lemma~\ref{lem-2-6}, (\ref{lem-4-2-1}) also
implies~(\ref{lem-4-2-3}).

In the proof of (\ref{lem-4-2-1}), we may assume without loss of
generality that $X$ is irreducible. The case ${\dim X \leq 1}$ is obvious,
$f$ being rational on algebraic curves. Suppose that ${d \coloneqq \dim X
\geq 2}$.

For each $i$, there exist a Zariski open dense subset $X_i^0 \subset X$
and a regular function $F_i$ on $X_i^0$ such that
\begin{equation}\label{eq-4-2-4}
f = F_i\quad \textrm{on}\ U_i \cap X_i^0.
\end{equation}
It remains to prove that for any $j$, the equality $F_i = F_j$ holds on
$X_i^0 \cap X_j^0$, which is equivalent to proving that $F_i = F_j$ on
$U_j \cap X_i^0 \cap X_j^0$. Suppose to the contrary that $F_i (x_0)
\neq F_j(x_0)$ for some $x_0 \in U_j \cap X_i^0 \cap X_j^0$. Then $F_i (x) \neq F_j(x)$
for all $x$ in an open neighborhood $U(x_0) \subset U_j
\cap X_i^0 \cap X_j^0$ of $x_0$. By Bertini's theorem, there exists an
irreducible real algebraic curve $C \subset X$ such that the
intersections $U(x_0) \cap C$ and $U_i \cap X_i^0 \cap C$ contain some
smooth algebraic arcs of $C(\R)$ (after fixing an embedding $X \subset
\PB^N$, such a curve $C$ is obtained by intersecting $X$ with a suitable
linear subspace $L \subset \PB^N$ of dimension $N - d -1$). Since $f$ is
rational on algebraic curves, there exists a regular function $G$
defined on a Zariski open dense subset $C^0 \subset C$ with
\begin{equation}\label{eq-4-2-5}
f = G \quad \textrm{on}\ U \cap C^0.
\end{equation}
From \eqref{eq-4-2-4} and \eqref{eq-4-2-5}, we get $F_i = G$ on $U_i \cap
X_i^0 \cap C^0$, which in turn implies that $F_i = G$ on $X_i^0 \cap
C^0$, hence
\begin{equation}\label{eq-4-2-6}
F_i = G\quad \textrm{on}\ U(x_0) \cap C^0.
\end{equation}
On the other hand, \eqref{eq-4-2-4} and \eqref{eq-4-2-5} also yield $F_j =
G$ on $U_j \cap X_j^0 \cap C^0$, hence
\begin{equation}\label{eq-4-2-7}
F_j = G\quad \textrm{on}\ U(x_0) \cap C^0.
\end{equation}
By \eqref{eq-4-2-6} and \eqref{eq-4-2-7}, $F_i = F_j$ on $U(x_0) \cap C^0$, a
contradiction.
\end{proof}

We now present the first application of Theorem~\ref{th-2-4}.

\begin{theorem}\label{th-4-3}
Let $X$ be a real algebraic variety, ${U \subset X(\R)}$ an open subset,
and $f \colon U \to \R$ a~function rational on algebraic curves. Assume
that $f$ is regular on smooth algebraic arcs. Then $f$ is hereditarily
rational.
\end{theorem}

\begin{proof}
It suffices to prove that $f$ is a rational function. By replacing $X$
with the Zariski closure of $U$ in $X$, we may assume that $U$ is
Zariski dense in $X$. Let $W$ be a connected component of $V \coloneqq U
\cap (X \setminus \Sing(X))$ and let $Y$ be the Zariski closure of $W$
in $X$. Then $W$ is a smooth open subset of $Y(\R)$. Clearly, $f|_W$ is
regular on smooth algebraic arcs. Thus, by Theorem~\ref{th-2-4}, $f|_W$
is a rational function. According to Lemma~\ref{lem-4-2}, $f|_V$ is a
rational function, hence $f$ is also a rational function.
\end{proof}

Theorem~\ref{th-2-4} allows us also to give the following
characterization of rational functions.

\begin{proposition}\label{prop-4-8}
Let $X$ be a real algebraic variety and let $U \subset X(\R)$ be a
connected smooth open subset. For a function $f \colon U \to \R$, the
following conditions are equivalent:
\begin{conditions}
\item\label{prop-4-8-a} $f$ is rational.
\item\label{prop-4-8-b} There exists a Zariski nowhere dense real
subvariety $Y \subset X$ such that for any smooth algebraic arc $A$
contained in $U$ the restriction $f|_{A \setminus Y}$ is a regular
function.
\end{conditions}
\end{proposition}

\begin{proof}
It is clear that (\ref{prop-4-8-a}) implies (\ref{prop-4-8-b}).

Suppose that (\ref{prop-4-8-b}) holds. By Theorem~\ref{th-2-4}, there
exist a nonempty open subset ${U_0 \subset U \setminus Y}$ and a rational
function $R$ on $X$ such that ${U_0 \subset X \setminus \Pole(R)}$ and
$f|_{U_0} = R|_{U_0}$. Set $P \coloneqq {Y \cup \Pole(R)}$. Proceeding as
in the proof of Lemma~\ref{lem-2-6}, we obtain that ${f|_{U \setminus P}
= R|_{U \setminus P}}$. Thus, (\ref{prop-4-8-b}) imp\-lies~(\ref{prop-4-8-a}).
\end{proof}

\subsection{Functions defined on semialgebraic subsets}

Next we consider functions defined on semialgebraic sets.

\begin{proposition}\label{prop-4-4}
Let $X$ be a real algebraic variety, $W \subset X(\R)$ a semialgebraic
subset, and $f \colon W \to \R$ a function regular on smooth algebraic
arcs. Then there exists a sequence of semialgebraic sets
\begin{equation*}
W = W_0 \supset W_1 \supset \ldots \supset W_m = \emptyset
\end{equation*}
such that, if $Y_i$ is the Zariski closure of $W_i$ in $X$, then for $i =
0, \ldots, m-1$ the following conditions hold:
\begin{nconditions}
\item\label{prop-4-4-1} $W_i = W \cap Y_i$;

\item\label{prop-4-4-2} $W_i \setminus W_{i+1}$ is a smooth open subset
of $Y_i(\R)$;

\item\label{prop-4-4-3} the restriction of $f$ is a regular function on
each connected component of $W_i \setminus W_{i+1}$;

\item\label{prop-4-4-4} $Y_i \setminus Y_{i+1}$ is Zariski dense in
$Y_i$.
\end{nconditions}
If, in addition, $f$ is rational on algebraic curves, then the
restrictions $f|_{W_i \setminus W_{i+1}}$ are regular functions.
\end{proposition}

\begin{proof}
Set $W_0 \coloneqq W$ and let $Y_0$ be the Zariski closure of $W_0$ in
$X$. We now describe how to construct $W_1$. To this end, set $M_0
\coloneqq Y_0(\R) \setminus \Sing(Y_0)$. Denote by $W_0^*$ the interior
of $W_0 \cap M_0$ in $M_0$. Then $W_0^*$ is a semialgebraic subset of
$X(\R)$, and the Zariski closure of $W_0 \setminus W_0^*$ in $X$ is
nowhere dense in $Y_0$ \cite[Chapter~2]{bib4}. Let $S_1$ be the Zariski
closure of $(W_0 \setminus W_0^*) \cup \Sing(Y_0)$ in~$X$. Setting $V_1
\coloneqq W_0 \cap S_1$, we see that the semialgebraic subset $W_0
\setminus V_1 \subset M_0$ is open in $M_0$. Clearly, $f|_{W_0 \setminus
V_1}$ is regular on smooth algebraic arcs. By Theorem~\ref{th-2-4}, the
restriction of $f$ to each connected component of $W_0 \setminus V_1$ is
a rational function. Since $W_0 \setminus V_1$ has finitely many
connected components, we can find a Zariski nowhere dense real
subvariety $Z_1 \subset Y_0$ such that $S_1 \subset Z_1$ and for each
connected component $K$ of $W_0 \setminus V_1$, the restriction $f|_{K
\setminus Z_1}$ is a regular function. Set $W_1 \coloneqq W_0 \cap Z_1$
and let $Y_1$ be the Zariski closure of $W_1$ in $X$. Then $Y_1 \subset
Z_1$, $W_1 = W _0\cap Y_1$, $W_0 \setminus W_1 \subset W_0 \setminus V_1$,
and $W_0 \setminus W_1$ is open in $M_0$. By construction, conditions
(\ref{prop-4-4-1}), (\ref{prop-4-4-2}), (\ref{prop-4-4-3}),
(\ref{prop-4-4-4}) hold for $i=0$.

Since the function $f|_{W_1}$ is regular on smooth algebraic arcs, we
can repeat this process to construct $W_2$, and so on. By
(\ref{prop-4-4-4}), $\dim Y_i > \dim Y_{i+1}$, hence we get $W_m =
\varnothing$ after finitely many steps, which proves (\ref{prop-4-4-1}),
(\ref{prop-4-4-2}), (\ref{prop-4-4-3}), (\ref{prop-4-4-4}) for all $i =
0, \ldots, m-1$.

If $f$ is also rational on algebraic curves, it follows from
(\ref{prop-4-4-2}), (\ref{prop-4-4-3}) and Lemma~\ref{lem-4-2} that the
$f|_{W_i \setminus W_{i+1}}$ are regular functions.
\end{proof}

\begin{theorem}\label{th-4-5}
Let $X$ be a real algebraic variety, ${W \subset X(\R)}$ a semialgebraic
subset, and \linebreak ${f \colon W \to \R}$ a function rational on algebraic curves.
Assume that $f$ is regular on smooth algebraic arcs. Then $f$ is
hereditarily rational.
\end{theorem}

\begin{proof}
It suffices to combine Propositions~\ref{prop-4-1} and~\ref{prop-4-4}.
\end{proof}

The following is an immediate consequence of Proposition~\ref{prop-4-4}.

\begin{corollary}\label{cor-4-6}
Let $X$ be a real algebraic variety, ${W \subset X(\R)}$ a semialgebraic
subset, and \linebreak ${f \colon W \to \R}$ a function regular on smooth algebraic
arcs. Then $f$ is a semialgebraic function. \qed
\end{corollary}

We now give a characterization of hereditarily rational functions
defined on semialgebraic sets.

\begin{theorem}\label{th-4-7}
Let $X$ be a real algebraic variety and let $W \subset X(\R)$ be a
semialgebraic subset. For a function $f \colon W \to \R$, the following
conditions are equivalent:
\begin{conditions}
\item\label{th-4-7-a} $f$ is hereditarily rational.

\item\label{th-4-7-b} $f$ is rational on algebraic curves and
semialgebraic.
\end{conditions}
\end{theorem}

\begin{proof}
If (\ref{th-4-7-a}) holds, then $f$ is rational on algebraic curves.
Furthermore, $f$ is semialgebraic by Proposition~\ref{prop-4-1}. Thus,
(\ref{th-4-7-a}) implies (\ref{th-4-7-b}).

Suppose that (\ref{th-4-7-b}) holds. To prove (\ref{th-4-7-a}), it
suffices to show that $f$ is a rational function. Let~$Y$ be the Zariski
closure of $W$ in $X$. Since $f$ is semialgebraic, there exists a
Zariski open dense subset $Y^0 \subset Y$ such that the restriction
$f|_{W \cap Y^0}$ is continuous \cite{bib4}. The function
$f|_{W \cap Y^0}$ is also rational on algebraic curves. It follows that
$f|_{W \cap Y^0}$ is regular on smooth algebraic arcs. Thus, by
Theorem~\ref{th-4-5}, $f|_{W \cap Y^0}$ is rational, which means that
$f$ is rational as well.
\end{proof}

\section{Arc-rational functions}\label{sec-5}
\subsection{Continuity}
First we address continuity of arc-rational functions.

\begin{proposition}\label{prop-5-1}
Let $X$ be a real algebraic variety and let $f \colon W \to \R$ be an
arc-rational function defined on a subset $W \subset X(\R)$ that is
either open or semialgebraic. Then $f$ is continuous.
\end{proposition}

\begin{proof}
Since continuity is a local property, it suffices to consider $W$
semialgebraic. Then $f$ is a semialgebraic function by
Corollary~\ref{cor-4-6}. We now prove continuity of $f$ at an arbitrary
point $x_0 \in W$.

We fix an embedding $\R \subset \PB^1(\R)$ and regard $\Gamma(f)$, the
graph of $f$, as a subset of ${X(\R) \times \PB^1(\R)}$. Let $l \in
\PB^1(\R)$ be any point such that $(x_0, l)$ belongs to the closure of
$\Gamma(f)$. It remains to prove that $f(x_0) = l$. By the Nash curve
selection lemma \cite[Proposition~8.1.13]{bib4}, there exists a Nash arc
$\varphi = (\gamma, \psi) \colon (-1,1) \to X(\R) \times \PB^1(\R)$ with
\begin{equation*}
\varphi(0) = (\gamma(0), \psi(0)) = (x_0, l)\quad \textrm{and}\quad
\varphi((0,1)) \subset \Gamma(f).
\end{equation*}
In particular,
\begin{equation*}
\psi(t) = f(\gamma(t))\quad \textrm{for}\ t \in (0,1).
\end{equation*}
Let $C \subset X$ be the Zariski closure of the semialgebraic set
$\gamma((-1,1))$. Then, either $C = \{x_0\}$ or $C$ is an irreducible
real algebraic curve with $x_0 \in C(\R)$. Since $f$ is arc-rational,
the restriction $f|_{W \cap C}$ is continuous. Consequently, the
function $f \circ \gamma$, which is well defined on $[0, 1)$, is
continuous at $0$, hence
\begin{equation*}
\lim_{t \to 0^+} f(\gamma(t)) = f(x_0).
\end{equation*}
On the other hand,
\begin{equation*}
\lim_{t \to 0} \psi(t) = l.
\end{equation*}
It follows that $f(x_0) = l$, as required.
\end{proof}

\begin{proof}[Proof of Theorem~\ref{th-1-7}]
It is clear that (\ref{th-1-7-a}) implies (\ref{th-1-7-b}).

Suppose that (\ref{th-1-7-b}) holds. Then $f$ is regular on smooth
algebraic arcs. Hence, according to Theorems~\ref{th-4-3}
and~\ref{th-4-5}, $f$ is hereditarily rational. Furthermore, $f$ is
arc-rational, hence continuous by Proposition~\ref{prop-5-1}. Thus
(\ref{th-1-7-a}) holds.
\end{proof}

\begin{proposition}\label{prop-5-2}
Let $X$ be a real algebraic variety, $U \subset X(\R)$ an open subset,
and ${f \colon U \to \R}$ a~continuous rational function.\ Then, for each
irreducible real subvariety ${Z \subset X}$ with \linebreak ${U \cap (X \setminus
\Sing(X)) \cap (Z \setminus \Sing(Z)) \neq \varnothing}$, the restriction
$f|_{U \cap Z}$ is a rational function.
\end{proposition}

\begin{proof}
One can repeat the proof of \cite[Proposition~8]{bib11} with only minor
modifications.
\end{proof}

\begin{proof}[Proof of Theorem~1.10]
It is clear that (\ref{th-1-10-a}) implies (\ref{th-1-10-b}).

Suppose that (\ref{th-1-10-b}) holds. According to Theorem~\ref{th-2-4}
and Proposition~\ref{prop-5-1}, $f$ is continuous rational. Let $Z
\subset X$ be a real subvariety with $U \cap Z \neq \varnothing$ and let
$Y$ be an irreducible component of the Zariski closure of $U \cap Z$ in
$X$. Then $U \cap (Y \setminus \Sing(Y)) \neq \varnothing$, hence $f|_{U
\cap Y}$ is a rational func\-tion by Proposition~\ref{prop-5-2}.
Consequently, $f|_{U \cap Z}$ is a rational function. Thus,
(\ref{th-1-10-b}) implies~(\ref{th-1-10-a}).
\end{proof}
\begin{proof}[Proof of Theorem~1.11]
It suffices to combine Propositions \ref{prop-4-4} and \ref{prop-5-1}
\end{proof}

\subsection{Arc-analyticity}

We now prepare to deal with arc-analyticity of arc-rational functions.

\begin{lemma}\label{lem-5-3}
Let $X$ be a real algebraic variety and let $f \colon W \to \R$ be an
arc-rational function defined on a subset $W \subset X(\R)$. For some
$\varepsilon > 0$, let $\gamma \colon (-\varepsilon, \varepsilon) \to
X(\R)$ be a Nash arc with $\gamma((-\varepsilon, \varepsilon)) \subset
W$. Then $f \circ \gamma$ is an analytic function.
\end{lemma}

\begin{proof}
It is harmless to assume that $\gamma$ is not a constant map. Let $C$ be
the Zariski closure in $X$ of the semialgebraic set
$\gamma((-\varepsilon, \varepsilon))$. Then $C$ is an irreducible real
algebraic curve.

Fixing $t_0 \in (-\varepsilon, \varepsilon)$ and setting $x_0 \coloneqq
\gamma(t_0)$, we can find an open neighborhood $U(x_0) \subset W$ of
$x_0$ such that the restriction $f|_{U(x_0) \cap C}$ is a continuous
rational function. Then $I(t_0) \coloneqq \gamma^{-1}(U(x_0) \cap C)$ is
an open neighborhood of $t_0$ in $(-\varepsilon, \varepsilon)$, and $(f
\circ \gamma)|_{I(t_0)}$ is a continuous meromorphic function. Any
continuous real meromorphic function on an open interval is analytic. It
follows that $f \circ \gamma$ is analytic on $(-\varepsilon,
\varepsilon)$.
\end{proof}

\begin{proof}[Proof of Theorem~\ref{th-1-11}]
In view of  Proposition~\ref{prop-5-1}, it remains to prove that $f$ is
arc-analytic. Since each point in
$W$ has a semialgebraic open neighborhood, we may assume without loss of
generality that $W$ is open and semialgebraic. Then, by
Corollary~\ref{cor-4-6} and Proposition~\ref{prop-5-1}, $f$ is a
continuous semialgebraic function.

It is convenient to assume for the rest of the proof that $X \subset
\AB^N$, hence $X(\R) \subset \R^N$. This is justified since $X$ can be
replaced with an affine open subset containing $X(\R)$. As a consequence
of the \L{}ojasiewicz inequality \cite[Theorem~2.6.6]{bib4}, we obtain
that the function $f$ is locally H\"older. More precisely, for any point
$x \in W$ we can find an open neighborhood $U \subset W$ and two
constants $\rho > 0$, $C > 0$ such that
\begin{equation}\label{eq-1-11-1}
|f(y) - f(y')| \leq C|y - y'|^{\rho}\quad \textrm{for all}\ y, y' \in U.
\end{equation}

In order to complete the proof, it suffices to show that for each
analytic arc ${\eta \colon (-1, 1) \to W}$, the function $f \circ \eta$ is
analytic at $0 \in (-1, 1)$. Since $f \circ \eta$ is continuous and
 subanalytic (in fact it is semianalytic by a  result of {\L}ojasiewicz, cf. \cite{bib21}), it has near $0$ two (possibly distinct) expansions as
convergent Puiseux series. This means that for some integer $r > 0$,
\begin{equation*}
f(\eta(t)) = \sum_{i=1}^{\infty} a_i t^{i/r}\quad \textrm{for}\ 0 \leq t \ll
1\quad \textrm{and}\quad f(\eta(t)) = \sum_{i=0}^{\infty} b_i
(-t)^{i/r}\quad \textrm{for}\ -1 \ll t \leq 0,
\end{equation*}
where the Puiseux series are convergent \cite[Corollary~2.7]{bib21}.
Thus, $f \circ \eta$ is analytic at $t = 0$ if and only if
\begin{equation}\label{eq-1-11-2}
a_i = 0 = b_i\quad \textrm{for}\ i \in (\N \setminus r\N)\quad
\textrm{and}\quad a_i = (-1)^i b_i\quad \textrm{for}\ i \in r\N.
\end{equation}
Suppose that $f \circ \eta$ is not analytic at $t=0$, hence at least one
of the conditions in \eqref{eq-1-11-2} is violated. It follows that
there exists an integer $k > 0$ with the following properties: for every
$\varepsilon > 0$ and every analytic function $h \colon (-\varepsilon,
\varepsilon) \to \R$ we can find a constant $c_h > 0$ such that
\begin{equation}\label{eq-1-11-3}
\begin{aligned}
\textrm{either}\ |f(\eta(t)) - h(t)| &> c_h t^k\ &&\textrm{for}\ 0 < t \ll
\varepsilon\\
\textrm{or}\ |f(\eta(t)) - h(t)| &> c_h |t|^k\ &&\textrm{for}\
-\varepsilon \ll t < 0.
\end{aligned}
\end{equation}
However, this leads to a contradiction. Indeed, by
\cite[Lemma~2.9]{bib22}, or alternatively, by Artin's approximation theorem \cite[Theorem~8.3.1]{bib4}, for every integer $s > 0$ there exist constants
$\varepsilon > 0$, $c > 0$ and a Nash arc $\gamma \colon (-\varepsilon,
\varepsilon) \to W$ satisfying
\begin{equation}\label{eq-1-11-4}
|\eta(t) - \gamma(t)| \leq c|t|^s\quad \textrm{for}\ |t| \ll \varepsilon.
\end{equation}
Choose $s$ and $\gamma$ so that \eqref{eq-1-11-4} holds and $s \rho >
k$. In view of \eqref{eq-1-11-1}, with $x = \eta(0) = \gamma(0)$, we get
\begin{equation*}
|f(\eta(t)) - f(\gamma(t))| \leq C|\eta(t) - \gamma(t)|^{\rho} \leq cC
|t|^{s\rho}\quad \textrm{for}\ |t| \ll \varepsilon,
\end{equation*}
which contradicts \eqref{eq-1-11-3} since $f \circ \gamma$ is an analytic
function by Lemma~\ref{lem-5-3}. The proof is complete.
\end{proof}


\section{Rational functions on products}\label{sec-7}
\subsection{The main result}
Let $X = X_1 \times \cdots \times X_n$ be the product of real
algebraic varieties $X_1, \ldots, X_n$ and let $\pi_i \colon X \to
X_i$ be the projection on the $i$th factor. We say that a subset
$K \subset X(\R) = X_1(\R) \times \cdots X_n(\R)$ is
\emph{parallel to the $i$th factor of $X$} if $\pi_j(K)$ consists
of one point for each $j \neq i$.

The following is the main result of this section.

\begin{theorem}\label{th-7-1}
Let $X = X_1 \times \cdots \times X_n$ be the product of real
algebraic varieties and let $f \colon U \to \R$ be a function
defined on a connected smooth open subset $U \subset X(\R)$.
Assume that the restriction of $f$ is regular on each smooth
algebraic arc contained in $U$ and parallel to one of the factors
of~$X$. Then there exists a rational function $R$ on $X$ such
that $P \coloneqq U \cap \Pole(R)$ has codimension at least~$2$
and $f|_{U \setminus P} = R|_{U \setminus P}$.
\end{theorem}

Theorem~\ref{th-7-1}, for $n=1$, coincides with
Theorem~\ref{th-2-4}. We prove the general case by induction on
$n$, but this requires some preparation.
First, however, we give the following immediate consequence of
Theorem~\ref{th-7-1}.

\begin{corollary}\label{cor-7-2}
Let $f \colon U \to \R$ be a function defined on a connected open
subset $U \subset \R^n$. Assume that the restriction of $f$ is a
regular function on each open interval contained in $U$ and
parallel to one of the coordinate axes. Then there exists a
rational function $R$ on $\mathbb A^n$ such that $P \coloneqq U \cap
\Pole(R)$ has codimension at least $2$ and $f|_{U \setminus P} =
R|_{U \setminus P}$. \qed
\end{corollary}

There is a related result in \cite{bib8}, which however is of a somewhat
different nature. We recall it below.

Let $f \colon U \to \R$ be an analytic function defined on a
connected open subset $U \subset \R^n$. Assume that the
restriction of $f$ is rational on each open interval $I$
contained in $U$ and parallel to one of the coordinate axes.
According to \cite[p.~201, Theorem~5]{bib8}, $f$ is then a
rational function. Since $f$ is assumed to be analytic, the
restriction $f|_I$ is regular. Furthermore, it easily follows
that $f$ is regular on $U$ (see the beginning of  Section \ref{sec-8}).


In Corollary~\ref{cor-7-2}, the function $f$ need not be regular.
In fact, Example~\ref{ex-2-3} shows that much stronger
conditions do not imply even local boundedness of $f$.

Furthermore, the function $f$ in Corollary~\ref{cor-7-2} need not
be regular on smooth algebraic arcs.

\begin{example}\label{ex-7-3}
The function $f \colon \R^2 \to \R$, defined by
\begin{equation*}
f(x,y) = \frac{xy}{x^4 + y^4}\quad \textrm{for}\ (x,y) \neq (0,0)
\quad \textrm{and}\quad f(0,0)=0,
\end{equation*}
is regular on any line parallel to one of the coordinate axes but
it is not continuous on $y=x$.
\end{example}

Corollary~\ref{cor-7-2} could suggest that on an algebraic
surface any grid formed by $2$ pencils of algebraic curves may be
enough to check rationality. The next example shows that this is
not at all the case.

\begin{example}\label{ex-7-4}
On $\R^2$ consider the function $f(x,y) = \sqrt{x^2 + y^2 +1}$.
It is not rational on any open set.

For $a>0$ let $C_a$ be the hyperbola with the equation $ax^2 -
y^2 = 1$. Note that
\begin{equation*}
(x^2 + y^2 +1)|_{C_a} = (1+a)x^2|_{C_a},
\end{equation*}
hence $f|_{C_a} = \sqrt{1+a}|x|$ is a rational function on both
connected components of $C_a(\R)$. These hyperbolas form a
pencil, a unique one passing thought every point not  on the $y$-axis.

Note further that $f$ is invariant under rotations. We
can thus rotate the hyperbolas to get infinitely many pencils of
curves, together forming a $2$-parameter family ${\{C_{a, \theta}
\colon a>0,\ 0\leq \theta \leq 2\pi\}}$, such that the restriction
of $f$ to any one  of these curves is rational (even regular) on both
connected components of $C_{a, \theta}(\R)$. Through any given
point there is now a $1$-parameter family of rotated hyperbolas
$\{C_{\lambda}\}$ such that the $f|_{C_{\lambda}}$ are rational.

More generally, if $B \subset \mathbb A^2$ is a curve of degree~$d$
that is tangent to the conic $(x^2 + y^2 + 1 =0)$ at $d$~points
then $(x^2 + y^2 +1)|_B$ is a square (over~$\CB$) hence
$f|_{B(\R)}$ is the absolute value of a rational function. The
family of such curves has dimension
\begin{equation*}
\binom{d+2}{2} - d -1.
\end{equation*}
Thus we get larger and larger  families of curves on which $f$ has rational
restriction.
\end{example}

\subsection{Initial steps}

Our first step towards the proof of Theorem~\ref{th-7-1} is  the
following variant of Lemma~\ref{lem-2-6}.

\begin{lemma}\label{lem-7-5}
Let $X = X_1 \times \cdots \times X_n$ be the product of real
algebraic varieties, $U \subset X(\R)$ a connected smooth open
subset, and $f \colon U \to \R$ a function whose restriction is
regular on each smooth algebraic arc contained in $U$ and
parallel to one of the factors of $X$. Assume that there exist a
nonempty open subset $U_0 \subset U$ and a rational function $R$
on $X$ such that $U_0 \subset U \setminus \Pole(R)$ and $f|_{U_0}
= R|_{U_0}$. Then $P \coloneqq U \cap \Pole(R)$ has codimension
at least $2$ and $f|_{U \setminus P} = R|_{U \setminus P}$.
\end{lemma}

\begin{proof}
If $\dim X \leq 1$, then $f$ is a rational function, hence the
assertion holds. Suppose that $\dim X \geq 2$. The Zariski
closure of $U$ in $X$ is an irreducible component of $X$, so we
may assume that $X$ is irreducible.

First we prove that
\begin{equation}\label{eq-7-5-1}
f|_{U \setminus P} = R|_{U \setminus P}.
\end{equation}
Let $\A$ be the set of all smooth algebraic arcs contained in
$U$ and parallel to one of the factors of $X$. Each point $p \in
U$ has an arbitrarily small open neighborhood $U(p) \subset U$ of
the form $U(p) = U_1(p) \times \cdots \times U_n(p)$, where
$U_i(p) \subset X_i(\R)$ is an open subset such that any two points
of $U_i(p)$ belong to a smooth algebraic arc contained in
$U_i(p)$ (see the proof of Lemma~\ref{lem-2-6}). Fix a point $p_0
\in U_0$ and let $p \in U \setminus P$ be an arbitrary point. Let
$\gamma \colon [0,1] \to U$ be a continuous path with $\gamma(0)
= p_0$ and $\gamma(1)=p$. We can cover the compact set
$\gamma([0,1])$ by a finite collection of open sets $U(p_0),
U(p_1), \ldots, U(p_r)$ such that $U(p_0) \subset U_0$, $p_r = p$, and
the intersection $U(p_i) \cap U(p_{i+1})$ is
nonempty for all $i = 0, \ldots, r-1$. Now we use double
induction to prove that
\begin{equation}\label{eq-7-5-2}
f|_{U(p_i) \setminus P} = R|_{U(p_i) \setminus P}
\end{equation}
for $i = 0, 1, \ldots, r$. Equality~\eqref{eq-7-5-2} is obvious
for $i=0$. Suppose that \eqref{eq-7-5-2} holds for $i=j$, where $0
\leq j < r$. Set $V \coloneqq U(p_{j+1})$ and define $V(k)$
recursively:
\begin{gather*}
V(0) \coloneqq U(p_j) \cap U(p_{j+1}), \\
V(k) \coloneqq \rho_k^{-1}(\rho_k(V(k-1))) \cap V \quad \textrm{for} \ 1
\leq k \leq n,
\end{gather*}
where
\begin{equation*}
\rho_k \colon X \to X_1 \times \cdots \times X_{k-1} \times
X_{k+1} \times \cdots \times X_n
\end{equation*}
is the canonical projection. Clearly,
\begin{equation*}
V(0) \subset V(1) \subset \ldots \subset V(n) = V
\end{equation*}
are open subsets. Moreover, \eqref{eq-7-5-2} holds for $i = j+1$
if and only if
\begin{equation}\label{eq-7-5-3}
f|_{V(k) \setminus P} = R|_{V(k) \setminus P}
\end{equation}
for $k = 0, 1, \ldots, n$. Equality \eqref{eq-7-5-3} is obvious
for $k=0$. Suppose that \eqref{eq-7-5-3} holds for $k=l$, where
$0 \leq l < n$, and let $x \in V(l+1) \setminus P$ be an
arbitrary point. Pick a point $x_0 \in V(l) \setminus P$ for
which $\rho_{l+1}(x_0) = \rho_{l+1}(x)$. Then there exists an arc
$A \in \A$ such that $A \subset V$ and $x_0, x \in A$. The
functions $f|_{A \setminus P}$, $R|_{A \setminus P}$ are regular
and equal on the nonempty open subset $V(l) \cap (A \setminus P)$
of $A$, hence $f|_{A \setminus P} = R|_{A \setminus P}$ and $f(x)
= R(x)$. Thus, \eqref{eq-7-5-3} holds for $k = l+1$.
Consequently, \eqref{eq-7-5-3} holds for $k = 0, 1, \ldots, n$,
and \eqref{eq-7-5-2} holds for $i = j+1$. The double induction is
complete, which means that \eqref{eq-7-5-2} is established for $i
= 0, 1, \ldots, r$. Equality~\eqref{eq-7-5-1} follows.

It remains to prove that $\codim P \geq 2$. Suppose to the
contrary that $\codim P = 1$. Let $B$ be an arc in $\A$ that
meets $P$ transversally at a general point. Then $f|_B$ is a
regular function satisfying
\begin{equation*}
(f|_B)|_{B \setminus P}  = f|_{B \setminus P} = R|_{B \setminus
P},
\end{equation*}
which means that $R$ cannot have a pole along $B$, a
contradiction.
\end{proof}

We need another auxiliary result.

\begin{lemma}\label{lem-7-6}
Let $X = X_1 \times \cdots \times X_n$ be the product of real
algebraic varieties and let $f \colon U \to \R$ be a rational
function defined on a connected smooth open subset $U \subset
X(\R)$. Assume that the restriction of $f$ is regular on each
smooth algebraic arc contained in $U$ and parallel to one of the
factors of $X$. If $U_0 \subset U$ is a nonempty open subset and
$f$ vanishes on some dense subset of $U_0$, then $f$ is
identically equal to $0$ on $U$.
\end{lemma}

\begin{proof}
By replacing $X$ with the Zariski closure of $U$ in $X$, we may
assume that $X$ is irreducible. Since $f$ is rational, there
exist a Zariski open dense subset $X^0 \subset X$ and a regular
function $F$ on~$X^0$ such that $f|_{U \cap X^0} = F|_{U \cap
X^0}$.


It follows that $F$ vanishes on  $X^0$, hence $f$ vanishes on $U \cap X^0$.
It remains to prove that
\begin{equation}\label{eq-7-6-1}
f(p)=0
\end{equation}
for an arbitrary point $p\in U$.  The argument is analogous to that used  in the proof 
Lemma \ref{lem-7-5}. We choose a neighborhood $V\subset U$ of $p$ of the form
$V= U_1\times \cdots \times U_n$, where $U_i \subset X_i(\R)$ is an open subset such 
any two points of $U_i$ belong to a smooth algebraic arc contained in $U_i$.
Define $V(k)$ recursively:
$$
V(0):= V\cap X^0,
$$
$$
V(k):= \rho_k^{-1}(\rho_k(V(k-1)))\cap V 
$$
for $1\le k \le n$, where
$$
 \rho_k: X \to X_1\times \cdots \times X_{k-1}\times X_{k+1}\times \cdots \times X_n
 $$
 is the canonical projection. Then
 $$
 V(0) \subset V(1) \subset \cdots \subset V(n)= V
 $$
 are open sets. Equality \eqref{eq-7-6-1} holds if 
 \begin{equation}\label{eq-7-6-2}
f|_{V(k)}=0
\end{equation}
for $k= 0,1,\dots, n$. Equality \eqref{eq-7-6-2} is obvious  for $k=0$. 
Suppose that  \eqref{eq-7-6-2} holds for $k=l$, where $0\le l <n$, and let $x\in V(l+1)$ be an arbitrary point.
Pick a point  $x_0 \in V(l)$ for which
$\rho_{l+1}(x_0)=\rho_{l+1}(x)$. Then there exists  a smooth algebraic arc $A\subset V$ that is parallel to one of the factors of $X$
and contains both points $x_0$ and $x$.
The function $f |_A$  is regular and equal to $0$ on the nonempty open subset 
$V(l)\cap A$, hence $f |_A=0$ and $f(x)=0$. Thus, \eqref{eq-7-6-2}  holds for $k=l+1$. By induction, \eqref{eq-7-6-2}  holds for 
 $k= 0,1,\dots, n$, which completes the proof.

\end{proof}

\subsection{Separately  rational functions}

The next two lemmas are refinements of \cite[pp.~199--201]{bib8}.

\begin{lemma}\label{lem-7-7}
Let $X$, $Y$ be real algebraic varieties and let $V \subset
X(\R)$, $W \subset Y(\R)$ be smooth open subsets. Assume that
$Y$ is affine and $W$ is connected. Let $f_1 \colon V \times W
\to \R, \ldots, f_r \colon V \times W \to \R$ be functions such
that for $i = 1, \ldots, r$ and each point $x \in V$ the function
\begin{equation*}
W \to \R,\quad y \mapsto f_i(x,y)
\end{equation*}
is regular on smooth algebraic arcs and it is not identically
equal to $0$. Assume that there are functions $\varphi_1 \colon
\Omega \to \R, \ldots, \varphi_r \colon \Omega \to \R$, defined
on some dense subset $\Omega \subset W$, for which the following
two conditions hold:
\begin{gather}
\label{eq-7-7-1}
\sum_{i=1}^r \varphi_i(y) f_i(x,y) = 0 \quad \textrm{for all}\
x \in V,\ y \in \Omega;\\
\label{eq-7-7-2}
\sum_{i=1}^r \varphi_i(y)^2 > 0 \quad \textrm{for all} \ y \in
\Omega.
\end{gather}
Then there exist regular functions $\Phi_1, \ldots, \Phi_r$ on $Y$
such that
\begin{equation}\label{eq-7-7-3}
\sum_{i=1}^r \Phi_i(y) f_i(x,y) = 0 \quad \textrm{for all}\ x\in
V,\ y \in W
\end{equation}
and the restrictions $\Phi_i|_W$ are not all identically equal to
$0$.
\end{lemma}

\begin{proof}
For $x_1, \ldots, x_r$ in $V$ and $y$ in $W$, consider the matrix
\begin{equation*}
\left[
\begin{array}{ccc}
f_1(x_1,y) & \ldots & f_r(x_1,y)\\
\vdots & & \vdots\\
f_1(x_r,y) & \ldots & f_r(x_r,y)\\
\end{array}
\right].
\end{equation*}
Conditions \eqref{eq-7-7-1} and \eqref{eq-7-7-2} imply that the
determinant of this matrix is equal to $0$ for all ${x_1, \ldots, x_r \in
V}$ and $y \in \Omega$. Thus, substituting $x$ for $x_r$ and expanding
the determinant along the last row, we get a relation
\begin{equation}\label{eq-7-7-4}
\sum_{i=1}^r h_i(x_1, \ldots, x_{r-1}; y) f_i(x,y) = 0 \quad \textrm{for
all} \ x_1, \ldots, x_{r-1} \in V, \ y \in \Omega,
\end{equation}
where $h_i(x_1, \ldots, x_{r-1}; y)$ is equal to $(-1)^{i+r}$ multiplied
by the determinant of the matrix obtained from
\begin{equation*}
\left[
\begin{array}{ccc}
f_1(x_1,y) & \ldots & f_r(x_1,y)\\
\vdots & & \vdots\\
f_1(x_{r-1},y) & \ldots & f_r(x_{r-1},y)\\
\end{array}
\right]
\end{equation*}
by deleting the $i$th column.

We complete the proof in two steps.

\begin{case}\label{case-7-7-1}
Suppose that for some specific points $x_1 = a_1, \ldots, x_{r-1} =
a_{r-1}$ in $V$ the functions
\begin{equation*}
h_i \colon W \to \R, \quad h_i(y) \coloneqq h_i(a_1, \ldots, a_{r-1}; y)
\end{equation*}
are not all identically equal to $0$; say, $h_j$ is not identically $0$.

By construction, the functions $h_i$ are regular on smooth algebraic
arcs, hence they are rational functions by Theorem~\ref{th-2-4}. Since
$Y$ is an affine variety, we can find regular functions $F_i$, $G_i$
on~$Y$ such that $F_i|W = (G_i|_W)h_i$ and the zeros of $G_i$ are contained
in a Zariski nowhere dense subvariety $Z \subset Y$. Setting $\Phi_i
\coloneqq G_1 \cdots G_{i-1} F_i G_{i+1} \cdots G_r$, we get from
\eqref{eq-7-7-4} a relation
\begin{equation}\label{eq-7-7-5}
\sum_{i=1}^r \Phi_i(y) f_i(x,y) = 0 \quad \textrm{for all} \ x \in V, \
y \in \Omega.
\end{equation}
For each point $x \in V$, the function
\begin{equation*}
W \to \R, \quad y \mapsto \sum_{i=1}^r \Phi_i(y) f_i(x,y)
\end{equation*}
is regular on smooth algebraic arcs, hence rational by
Theorem~\ref{th-2-4}. Thus, according to Lemma~\ref{lem-7-6}, the
equality in \eqref{eq-7-7-5} holds for all $x \in V$, $y \in W$.
Furthermore, applying Lemma~\ref{lem-7-6}, we obtain that $\Phi_j|W$ is
not identically equal to $0$. The proof of Case~\ref{case-7-7-1} is
complete.
\end{case}

\begin{case}\label{case-7-7-2}
Suppose that for $i =1, \ldots, r$ the equality $h_i(x_1, \ldots,
x_{r-1}; y) = 0$ holds for all $x_1, \ldots, x_{r-1} \in V$, $y \in W$.

Then, for some integer $s$ satisfying $2 \leq s < r$, there is a matrix
of the form
\begin{equation*}
\left[
\begin{array}{ccc}
f_{k_1}(x_{k_1},y) & \ldots & f_{k_s}(x_{k_1},y)\\
\vdots & & \vdots\\
f_{k_1}(x_{k_s},y) & \ldots & f_{k_s}(x_{k_s},y)\\
\end{array}
\right]
\end{equation*}
such that its determinant is identically equal to $0$ for all $x_{k_1},
\ldots, x_{k_s} \in V$, $y \in W$, but at least one minor of order $s-1$
of this matrix is not identically~$0$. Thus, by Case~\ref{case-7-7-1},
there exist regular functions $\Phi_{k_1}, \ldots, \Phi_{k_s}$ on $Y$
such that
\begin{equation*}
\sum_{l=1}^s \Phi_{k_l}(y) f_{k_l} (x,y) = 0 \quad \textrm{for all} \
x \in V, \ y \in W
\end{equation*}
and the restrictions $\Phi_{k_l}|_W$ are not all identically equal to
$0$. We obtain \eqref{eq-7-7-3} by inserting some $\Phi_i$ identically
equal to $0$.\qedhere
\end{case}%
\end{proof}

The last lemma plays a crucial role in the proof of the following.

\begin{lemma}\label{lem-7-8}
Let $X$, $Y$ be affine real algebraic varieties and let ${V \subset
X(\R)}$, ${W \subset Y(\R)}$ be con\-nected smooth open subsets. Let $f
\colon V \times W \to \R$ be a function with the following two
properties:
\begin{nconditions}
\item\label{lem-7-8-1} for each point $x \in V$ the function
\begin{equation*}
W \to \R, \quad y \mapsto f(x,y)
\end{equation*}
is regular on smooth algebraic arcs and it is not identically equal to
$0$;
\item\label{lem-7-8-2} for each point $y \in W$ the function
\begin{equation*}
V \to \R, \quad x \mapsto f(x,y)
\end{equation*}
is rational and it is not identically equal to $0$ on any nonempty open
subset of $V$.
\end{nconditions}
Then there exist nonempty open subsets $V_0 \subset V$, $W_0 \subset W$
such that the restriction $f|_{V_0 \times W_0}$ is a rational function.
\end{lemma}
\newpage

\begin{proof}
\setcounter{equation}{\value{nconditionsi}}
Since $X$ is an affine real algebraic variety, the $\R$-algebra $\OC(X)$ of regular functions on $X$ is finitely generated.
Let $\{ P_k \}$ be a basis for the $\R$-vector space $\OC(X)$.

According to (\ref{lem-7-8-2}), for each $y \in W$ we get a relation
\begin{equation}\label{eq-7-8-3}
\left(\sum_{i=1}^m \varphi_i(y) P_i)(x) \right) f(x,y) + \sum_{j=1}^n
\psi_j(y) P_j(x) = 0 \quad \textrm{for all} \ x \in V,
\end{equation}
where $m$, $n$ are positive integers depending on $y$, and the
$\varphi_i(y)$, $\psi_j(y)$ are real numbers satisfying
\begin{equation}\label{eq-7-8-4}
\sum_{i=1}^m \varphi_i(y)^2 + \sum_{j=1}^n \psi_j(y)^2 > 0.
\end{equation}
For each pair $(m,n)$ of positive integers, let $W(m,n)$ denote the set
of those $y \in W$ for which there is a relation~\eqref{eq-7-8-3} with
property~\eqref{eq-7-8-4}. Then
\begin{equation*}
W = \bigcup_{m,n} W(m,n),
\end{equation*}
hence by Baire's theorem, there exists a nonempty open subset $W_0
\subset W$ such that the intersection $W_0 \cap W(m,n)$ is dense in
$W_0$ for some $(m,n)$. Fix such a pair $(m,n)$ and set $\Omega
\coloneqq W_0 \cap W(m,n)$. We now have functions $\varphi_i \colon
\Omega \to \R$, $\psi_j \colon \Omega \to \R$ such that
\eqref{eq-7-8-3}~and~\eqref{eq-7-8-4} hold for all $y \in \Omega$.

The subset
\begin{equation*}
V_0 \coloneqq \{ x \in V \mid P_k(x) \neq 0 \ \textrm{for all} \ k=1,
\ldots, m+n \}
\end{equation*}
of $V$ is nonempty and open. Define functions
\begin{equation*}
f_1, \ldots, f_{m+n} \colon V_0 \times W_0 \to \R,
\end{equation*}
by $f_i(x,y) = P_i(x)f(x,y)$ for $i=1, \ldots, m$ and $f_{m+j}(x,y) =
P_j(x)$ for $j=1, \ldots, n$. By (\ref{lem-7-8-1}) and
Lemma~\ref{lem-7-7}, there exist regular functions $\Phi_i$, $\Psi_j$
on~$Y$ such that
\begin{equation*}
\left( \sum_{i=1}^m \Phi_i(y) P_i(x) \right) f(x,y) + \sum_{j=1}^n
\Psi_j(y) P_j(y) = 0 \quad \textrm{for all} \ x \in V_0, \ y \in W_0
\end{equation*}
and the restrictions $\Phi_i|_{W_0}$, $\Psi_j|_{W_0}$ are not all
identically equal to $0$. It follows that $f|_{V_0 \times W_0}$ is a
rational function.
\end{proof}

\begin{proof}[Proof of Theorem~\ref{th-7-1}]
By replacing $X_i$ with an affine open subset containing $X_i(\R)$, we
may assume that each $X_i$ is an affine variety.

We use induction on $n$. For $n=1$, Theorem~\ref{th-7-1} coincides with
Theorem~\ref{th-2-4}.

Suppose that $n \geq 2$. Let $U_1 \subset X_1(\R), \ldots, U_n \subset
X_n(\R)$ be connected smooth open subsets such that $U_1 \times \cdots
\times U_n \subset U$. By the induction hypothesis, for each point $x_n
\in U_n$ the function
\begin{equation*}
U_1 \times \cdots \times U_{n-1} \to \R, \quad (x_1, \ldots, x_{n-1})
\mapsto f(x_1, \ldots, x_{n-1}, x_n)
\end{equation*}
is rational. Furthermore, it is clear that this function is regular on
each smooth algebraic arc contained in $U_1 \times \cdots \times
U_{n-1}$ and parallel to one of the factors of $X_1 \times \cdots \times
X_{n-1}$. Consider the relation
\begin{equation}\label{eq-star}\tag{$\ast$}
f(x_1, \ldots, x_{n-1}, x_n) = 0.
\end{equation}
Denote by $\Delta$ the set of those points $x_n \in U_n$ for which
\eqref{eq-star} holds for all $(x_1, \ldots, x_{n-1})$ in ${U_1 \times
\cdots \times U_{n-1}}$. If $\Delta$ is dense in $U_n$, then $f$ is
identically equal to~$0$ on $U_1 \times \cdots \times U_{n-1} \times
U_n$ by Lemma~\ref{lem-7-6} and Theorem~\ref{th-2-4}.

Suppose that $\Delta$ is not dense in $U_n$ and let ${W \subset U_n}$ be a
connected open subset disjoint from~$\Delta$. Denote by $\Gamma$ the set
of those points $(x_1, \ldots, x_{n-1})$ in $U_1 \times \cdots \times
U_{n-1}$ for which \eqref{eq-star} holds for all~$x_n \in W$. If $\Gamma$
is dense in $U_1 \times \cdots \times U_{n-1}$, then $f$ is identically
equal to~$0$ on  $U_1 \times \cdots \times U_{n-1} \times W$ by
Lemma~\ref{lem-7-6}.

Suppose that $\Delta$ (resp.~$\Gamma$) is not dense in $U_n$ (resp.~$U_1
\times \cdots \times U_{n-1}$) and let $V \subset U_1 \times \cdots
\times U_{n-1}$ be a connected open subset disjoint from $\Gamma$. By
Lemma~\ref{lem-7-6}, the assumptions of Lemma~\ref{lem-7-8} are
satisfied for the function $f|_{V \times W}$. Hence the restriction
$f|_{V_0 \times W_0}$ is a rational function for some nonempty open
subsets $V_0 \subset V$, $W_0 \subset W$.

It follows from Lemma~\ref{lem-7-5} that Theorem~\ref{th-7-1} holds in
each of the cases considered above.
\end{proof}

\begin{remark} \label{rem-7-9} If in Theorem~\ref{th-7-1} $\dim X_i = 1$ for
$i=1,\dots, n$, then its proof given above does not depend on Theorem~\ref{th-2-4}.
In particular, one can obtain Corollary~\ref{cor-7-2} without making use of  Theorem~\ref{th-2-4}
and without involving analytic functions.
\end{remark}
\section{Regular functions}\label{sec-8}
As applications of Corollary~\ref{cor-7-2} and Theorem~\ref{th-4-3}, we
obtain two results on regular functions. For these applications we also
need a characterization of real analytic functions, recalled in
Theorem~\ref{th-8-3}.
For the sake of clarity, we state explicitly the following  simple observation.
If $X$ is a smooth real algebraic variety, $U\subset X(\R)$ an open subset, and $f: U\to \R$
a rational function which is analytic, then $f$ is a regular function \cite[Proposition~2.1]{bib12}.
This can also be seen  directly, since for each point $x\in U$, the ring of germs of real analytic functions at $x$ 
is faithfully flat over the ring of germs of regular functions at $x$.

\begin{theorem}\label{th-8-1}
Let $U \subset \R^n$ be a connected open subset, $n \geq 2$. For a
function $f \colon U \to \R$, the following conditions are equivalent:
\begin{conditions}
\item\label{th-8-1-a} $f$ is regular.

\item\label{th-8-1-b} For each $2$-dimensional affine plane $M \subset
\R^n$ the restriction of $f$ is regular on each connected component of
$U \cap M$.
\end{conditions}
\end{theorem}

\begin{proof}
It is clear that (\ref{th-8-1-a}) implies (\ref{th-8-1-b}).

Suppose that (\ref{th-8-1-b}) holds. Then the restriction of $f$ is a
regular function on each open interval contained in $U$. By
Corollary~\ref{cor-7-2}, $f$ is a rational function. Furthermore, $f$ is
also an analytic function according to Theorem~\ref{th-8-3}.
Consequently, $f$ is regular, being rational and analytic.
 Thus, (\ref{th-8-1-b})
implies (\ref{th-8-1-a}).
\end{proof}

For varieties we have the following.

\begin{theorem}\label{th-8-2}
Let $X$ be a smooth real algebraic variety of pure dimension $n \geq 2$
and let $U \subset X(\R)$ be an open subset. For a function $f \colon U
\to \R$, the following conditions are equivalent:
\begin{conditions}
\item\label{th-8-2-a} $f$ is regular.

\item\label{th-8-2-b} For each irreducible real algebraic surface ${S
\subset X}$ the restriction of $f$ is regular on\linebreak ${U \cap (S \setminus
\Sing(S))}$.
\end{conditions}
\end{theorem}

\begin{proof}
It is clear that (\ref{th-8-2-a}) implies (\ref{th-8-2-b}).

Suppose that (\ref{th-8-2-b}) holds.

First we prove that $f$ is rational. To this end consider an irreducible
real algebraic curve $C \subset X$ which has a smooth point $x \in
C(\R)$. We can find an irreducible real algebraic surface $S \subset X$
such that $C \subset S$ and $x$ is a smooth point of $S$. Condition
(\ref{th-8-2-b}) implies that $f|_{U \cap C}$ is regular at $x$.
Consequently, $f$ is rational on algebraic curves and regular on smooth
algebraic arcs. By Theorem~\ref{th-4-3}, $f$ is rational.

Next we show that $f$ is analytic. For each point $p \in U$, we can find
a Zariski open neighborhood $X(p) \subset X$, a real morphism $\varphi
\colon X(p) \to \AB^n$ and an open neighborhood 
$U(p) \subset U\cap X(p)$ such that $\varphi(U(p)) =
(-1, 1)^n \subset \R^n$ and the restriction $\psi \colon U(p) \to (-1,
1)^n$ of $\varphi$ is a real analytic diffeomorphism. Let $M \subset
\R^n$ be a $2$-dimensional affine plane and let $S$ be the Zariski
closure of $N \coloneqq \psi^{-1} ( (-1, 1)^n \cap M)$. Then $S \subset
X$ is an irreducible real algebraic surface and $N \subset S \setminus
\Sing(S)$. Condition (\ref{th-8-2-b}) implies that $f|_N$ is a regular
function, hence $(f \circ \psi^{-1})|_{(-1,1)^n \cap M}$ is an analytic
function. By Theorem~\ref{th-8-3}, $f \circ \psi^{-1}$ is an analytic
function. Consequently, $f|_{U(p)}$ is analytic, which means that so is
$f$.

In conclusion, $f$ is regular, being rational and analytic.
Thus, (\ref{th-8-2-b}) imp\-lies~(\ref{th-8-2-a}).
\end{proof}

We have used a result of Bochnak and Siciak \cite{bib5}, stated below as
Theorem~\ref{th-8-3}. Since \cite{bib5} is not easily available, we show
how to derive Theorem~\ref{th-8-3} from \cite{bib7}.

\begin{theorem}\label{th-8-3}
Let $f \colon U \to \R$ be a function defined on an open subset $U
\subset \R^n$, $n\geq 2$. Assume that for each $2$-dimensional affine
plane $M \subset \R^n$ the restriction $f|_{U \cap M}$ is an analytic
function. Then $f$ is analytic.
\end{theorem}

\begin{proof}
The assumption implies that for $x \in U$ and $k \in \N$, the function
\begin{equation*}
\delta_x^k f \colon \R^n \to \R, \quad
\delta_x^k f(v) \coloneqq \left. \frac{d^k}{dt^k} f(x+tv)\right|_{t=0}
\end{equation*}
is well defined. Furthermore, the restriction of $\delta_x^k f$ to any
$2$-dimensional vector subspace of $\R^n$ is a homogeneous polynomial of
degree~$k$. In particular, the restriction of $\delta_x^k f$ to any
affine line in~$\R^n$ is a polynomial function. By \cite[Lemma~1]{bib6}
the latter property implies that $\delta_x^k f$ is a polynomial
function. Thus, according to \cite[Theorem~7.5]{bib7}, $f$ is an
analytic function (in \cite[Thoerem~7.5]{bib7} it is also assumed that
$f$ is continuous, but this is required only for functions defined in
topological vector spaces of infinite dimension).
\end{proof}

\section{Rational functions on complex algebraic varieties}\label{sec-9}
\subsection{Preliminaries}
Throughout this section, a complex algebraic variety is a
quasi-projective variety $X$ defined over $\CB$ (always reduced but
possibly reducible). By a subvariety we mean a closed subvariety. We
regard $X(\CB)$, the set of complex points of $X$, as a complex analytic
variety. An open subset $U \subset X(\CB)$ is said to be \emph{smooth}
if it is contained in $X \setminus \Sing(X)$.

Our goal is to present complex counterparts of the results obtained in
the preceding sections. We restrict our attention to functions defined
on some open subset $W \subset X(\CB)$ since for complex varieties only
this case seems to be of interest.

We say that a function $f \colon W \to \CB$ is \emph{regular at a point
$x \in W$} if for some regular function~$\Phi_x$ defined on a Zariski
open neighborhood $X_x \subset X$ of $x$ the equality $f|_{W \cap X_x} =
\Phi_x|_{W \cap X_x}$ holds. Moreover, $f$ is said to be \emph{regular}
if it is regular at each point in $W$.

We say that $f$ is \emph{regular on smooth algebraic arcs} if the
restriction $f|_A$ is a regular function for each smooth algebraic arc
$A \subset W$. Here $A$ is called a \emph{smooth algebraic arc} if it is
a smooth open subset of $C(\CB)$, where $C \subset X$ is an irreducible
complex algebraic curve, that is homeomorphic to the unit open disc $\DB
\coloneqq \{ z \in \CB \colon |z| <1 \}$.

We say that $f$ is a \emph{rational function} if there exist a rational
function $R$ on $X$ and a Zariski open dense subset $X^0 \subset X$ such
that $X^0 \subset X \setminus \Pole(R)$ and $f|_{W \cap X^0} = R|_{W
\cap X^0}$. This is consistent with Definition~\ref{def-1-1} since the
Zariski closure of $W$ in $X$ is the union of some irreducible
components of~$X$.

We are mainly interested in continuous rational functions. By the
Riemann extension theorem, if $f$ is continuous rational and $X$ is a
normal variety, then $f$ is regular. However if $X$ is not normal it can happen that
$f$ is continuous rational but not regular.
\begin{example}\label{ex-9-1}
Let $X(\CB)= (z^2 - x^3y=0)$ and $f(x,y,z) = \frac{z^3}{x^4y}$. Then
$|f(x,y,z)|= |xy|^{1/2}$ on $X(\CB)$. Hence $f$ extends to a continuous rational function which is not regular at the origin.
\end{example}

 In general, continuous rational
functions are related to the notions of seminormality and
seminormalization; see \cite{bibA1, bibA2} or
\cite[Section~10.2]{bibA4}.

We say that $f$ is \emph{continuous rational on algebraic arcs} or
\emph{arc-rational} for short if for every point $x \in W$ and every
irreducible complex algebraic curve $C \subset X$, with $x \in C(\CB)$,
there exists an open neighborhood $U_x \subset W$ of $x$ such that the
function $f|_{U_x \cap C}$ is continuous rational.

\subsection{Complex regular functions}
The following corresponds to Theorem~\ref{th-2-4}.

\begin{theorem}\label{th-9-1}
Let $X$ be a complex algebraic variety, $U \subset X(\CB)$ a connected
smooth open subset, and $f \colon U \to \CB$ a function regular on
smooth algebraic arcs. Then $f$ is a regular function.
\end{theorem}

As in the real case, it is convenient to make use of Nash functions.
Recall that in the complex setting Nash maps can be defined as follows
\cite{bibA3}. Let $X$, $Y$ be complex algebraic varieties and let ${U
\subset X(\CB)}$, ${V \subset Y(\CB)}$ be open subsets. A map $\psi \colon
U \to V$ is said to be a \emph{Nash map} if it is holomorphic and each
point $x \in U$ has an open neighborhood $U_x \subset U$ such that the
graph of $f|_{U_x}$ is contained in a complex algebraic subvariety of $X
\times Y$ of dimension equal to $\dim U_x$ (this subvariety can be
chosen irreducible if $x$ is a smooth point of $X$ and $U_x$ is a
connected smooth open neighborhood of $x$). The composition of Nash maps
is a Nash map. In case $V = \CB$, we obtain \emph{Nash functions} (also
called \emph{algebraic functions}).

\begin{proof}[Proof of Theorem~\ref{th-9-1}]
We may assume that $X$ is irreducible and smooth. The case $\dim X \leq
1$ is obvious. Suppose that $d \coloneqq \dim X \geq 2$.
For each point ${x \in U}$, we can find a Zariski open neighborhood ${X_x
\subset X}$, a morphism
${\varphi_x \colon X_x \to \AB^d}$ and an open
neighborhood $U_x \subset U$ such that ${U_x \subset X_x}$,
${\varphi_x(U_x) = \DB^d \subset \CB^d}$ and the restriction $\psi_x
\colon U_x \to \DB^d$ of $\varphi_x$ is a Nash isomorphism. Then $f
\circ \psi_x^{-1} \colon \
\DB^d \to \CB$ is a function of $d$ complex variables $(z_1, \ldots,
z_d)$, which has the following property: for $j=1, \ldots, d$ and every
point $(a_1, \ldots, a_d) \in \DB^d$, the assignment
\begin{equation*}
\DB \to \CB, \quad z_j \mapsto (f \circ \psi_x^{-1}) (a_1, \ldots,
a_{j-1}, z_j, a_{j+1}, \ldots, a_d)
\end{equation*}
is a Nash function on $\DB$. By Hartogs' theorem, $f \circ \psi_x^{-1}$
is a holomorphic function on $\DB^d$. Furthermore, according to
\cite[p.~202, Theorem~6]{bib8}, $f \circ \psi_x^{-1}$ is a Nash function
on $\DB^d$. Consequently, $f|_{U_x}$ is a Nash function. Now, the
argument used in the proof of Proposition~\ref{prop-2-5} shows that
$f|_{U_x}$ is a rational function. Thus, $f|_{U_x}$ is a regular
function, being rational and holomorphic. It easily follows that $f$ is
a regular function.
\end{proof}

Let $X = X_1 \times \cdots \times X_n$ be the product of complex
algebraic varieties and let $\pi_i \colon X \to X_i$ be the canonical
projection. We say that a subset $K \subset X(\CB)$ is \emph{parallel to
the $i$th factors of $X$} if $\pi_j(K)$ consists of one point for each $j
\neq i$.

As a counterpart of Theorem~\ref{th-7-1}, we get the following.

\begin{theorem}\label{th-9-2}
Let $X = X_1 \times \cdots \times X_n$ be the product of complex
algebraic varieties and let $f \colon U \to \CB$ be a function defined
on a connected smooth open subset $U \subset X(\CB)$. Assume that the
restriction of $f$ is regular on each smooth algebraic arc contained in
$U$ and parallel to one of the factors of $X$. Then $f$ is a regular
function.
\end{theorem}

\begin{proof}
We use induction on $n$. For $n=1$, Theorem~\ref{th-9-2} coincides with
Theorem~\ref{th-9-1}.

Suppose that $n \geq 2$. By Hartogs' theorem, $f$ is a holomorphic
function. Clearly, each point $x \in U$ has a neighborhood in $U$ of the
form $U_1 \times \cdots \times U_n$, where $U_i \subset X_i(\CB)$ is a
connected smooth open subset that is contained in an affine open subset
of $X_i$. Setting $V \coloneqq U_1 \times \cdots \times U_{n-1}$, $W
\coloneqq U_n$, using the induction hypothesis and applying
Proposition~\ref{prop-9-3} below, we conclude that $f|_{V \times W}$ is
a regular function. Now it easily follows that $f$ is a regular
function.
\end{proof}

\begin{proposition}\label{prop-9-3}
Let $X$, $Y$ be affine complex algebraic varieties and let $V \subset
X(\CB)$, $W \subset Y(\CB)$ be connected smooth open subsets. Let $f
\colon V \times W \to \CB$ be a holomorphic function with the following
two properties:
\begin{nconditions}
\item\label{prop-9-3-1} for each point $x \in V$ the function
%
$W \to \CB, \quad y \mapsto f(x,y)$
%
is regular;
\item\label{prop-9-3-2} for each point $y \in W$ the function
%
$
V \to \CB, \quad x \mapsto f(x,y)
$
%
is regular.
\end{nconditions}
Then the function $f$ is regular.
\end{proposition}

\begin{proof}
This is a straightforward generalization of \cite[p.~201,
Thoerem~6]{bib8}. Of course, Proposition~\ref{prop-9-3} corresponds to
Lemma~\ref{lem-7-8}, but the proof in the complex setting is easier
because $f$ is holomorphic.
\end{proof}

Theorem~\ref{th-9-2} implies the following.

\begin{corollary}\label{cor-9-4}
Let $X = X_1 \times \cdots \times X_n$ be the product of complex
algebraic varieties, $X^0 \subset X$ a Zariski open subset, and $f \colon
X^0(\CB) \to \CB$ a function whose restriction is regular on each smooth
algebraic arc contained in $X^0(\CB)$ and parallel to one of the factors
of $X$. Then $f$ is a rational function.
\end{corollary}

\begin{proof}
We may assume that $X$ is irreducible, in which case $X^0(\CB) \setminus
\Sing(X)$ is a connected open subset of $X(\CB)$. Now it suffices to
apply Theorem~\ref{th-9-2}.
\end{proof}

\subsection{Continuous complex rational functions}
Next we prove continuity of arc-rational functions.

\begin{theorem}\label{th-9-5}
Let $X$ be a complex algebraic variety and let $f \colon W \to \CB$ be
an arc-rational function defined on an open subset $W \subset X(\CB)$.
Then $f$ is continuous.
\end{theorem}

\begin{proof}
We begin with a general remark. Identifying $\CB$ with $\R^2$, one can
consider semialgebraic subsets of $Y(\CB)$ for any complex algebraic
variety $Y$.

Since continuity is a local property, it suffices to consider $W$ open
and semialgebraic. Then $W_0 \coloneqq W \setminus \Sing(X)$ is also
semialgebraic. Furthermore, we may assume that $X$ is irreducible and
$\dim X \geq 1$. We now prove continuity of $f$ at an arbitrary point
$x_0 \in W$.

We fix an embedding $\CB \subset \PB^1(\CB)$ and regard $\Gamma(f)$, the
graph of $f$, as a subset of $X(\CB) \times \PB^1(\CB)$. Let $l \in
\PB^1(\CB)$ be any point such that $(x_0, l)$ belongs to the closure of
$\Gamma(f)$. It remains to prove that $f(x_0) = l$.

We claim that $\Gamma(f|_{W_0})$ is dense in $\Gamma(f)$. Indeed, for
any point $x \in W$ there exists an irreducible complex algebraic curve
$B \subset X$ with $x \in B$ and $W_0 \cap B \neq \varnothing$. Then $x$
belongs to the closure of $W_0 \cap B$ in $W \cap B$. Since $f$ is
arc-rational, the restriction $f|_{W \cap B}$ is a continuous function.
Hence $(x ,f(x))$ belongs to the closure of $\Gamma(f|_{W_0})$ in
$\Gamma(f)$, which proves the claim.

According to the claim, $(x_0, l)$ belongs to the closure of
$\Gamma(f|_{W_0})$. It follows from Theorem~\ref{th-9-1} that
$\Gamma(f|_{W_0})$ is a semialgebraic subset of $X(\CB) \times
\PB^1(\CB)$. Thus, by the Nash curve selection lemma
\cite[Proposition~8.1.13]{bib4}, there exists a Nash arc $\varphi =
(\gamma, \psi) \colon (-1,1) \to X(\CB) \times \PB^1(\CB)$ with
\begin{equation*}
\varphi(0) = (\gamma(0), \psi(0)) = (x_0, l) \quad \textrm{and} \quad
\varphi((0,1)) \subset \Gamma(f|_{W_0}).
\end{equation*}
In particular,
\begin{equation*}
\psi(t) = f(\gamma(t)) \quad \textrm{for} \ t \in (0,1).
\end{equation*}
Let $C \subset X$ be the Zariski closure of the semialgebraic set
$\gamma((-1,1))$. Then, either $C = \{ x_0 \}$ or $C$ is an irreducible
complex algebraic curve with $x_0 \in C(\CB)$. Since $f$ is
arc-rational, $f|_{W \cap C}$ is continuous. Consequently, the function
$f \circ \gamma$, which is well defined on $[0,1)$, is continuous at
$0$, hence
\begin{equation*}
\lim_{t \to 0^+} f(\gamma(t)) = f(x_0).
\end{equation*}
On the other hand,
\begin{equation*}
\lim_{t \to 0} \psi(t) = l.
\end{equation*}
It follows that $f(x_0) = l$, as required.
\end{proof}

We can now characterize continuous rational functions.

\begin{theorem}\label{th-9-6}
Let $X$ be a complex algebraic variety. For a function $f \colon X(\CB)
\to \CB$, the following conditions are equivalent:
\begin{conditions}
\item\label{th-9-6-a} $f$ is continuous  rational.

\item\label{th-9-6-b} $f$ is arc-rational.
\end{conditions}
\end{theorem}

\begin{proof}
It follows from Proposition~\ref{prop-9-7} below that (\ref{th-9-6-a})
implies (\ref{th-9-6-b}).

Suppose that (\ref{th-9-6-b}) holds. According to Theorem~\ref{th-9-5},
$f$ is continuous. If $X$ is irreducible, then the set $X(\CB)
\setminus \Sing(X)$ is connected, hence $f$ is rational by
Theorem~\ref{th-9-1}. Consequently, $f$ is also rational if $X$ is
reducible. Thus, (\ref{th-9-6-b}) implies (\ref{th-9-6-a}).
\end{proof}

\begin{proposition}\label{prop-9-7}
Let $X$ be a complex algebraic variety. For a function $f \colon X(\CB)
\to \CB$, the following conditions are equivalent:
\begin{conditions}
\item\label{prop-9-7-a} $f$ is continuous rational.
\item\label{prop-9-7-b} $f$ is continuous and $\Gamma(f) = (X(\CB)
\times \CB) \cap \Gamma$, where $\Gamma(f)$ is the graph of $f$ and
$\Gamma \subset X \times \AB^1$ is the Zariski closure of $\Gamma(f)$.
\end{conditions}
Furthermore, if these conditions hold, then for each algebraic
subvariety $Z \subset X$ the restriction $f|_{Z(\CB)}$ is a rational
function.
\end{proposition}

\begin{proof}
We will make use of the following familiar fact.
If $Y$ is a complex algebraic variety and $Y_0\subset Y$ is a Zariski open dense subset,
then the  set $Y_0 (\CB)$ is Euclidean dense in   $Y (\CB)$. 

Suppose that (\ref{prop-9-7-a}) holds.
 Then for some Zariski open dense
subset $X^0 \subset X$ the restriction $f|_{X^0(\CB)}$ is a regular
function. If $G \subset X \times \AB^1$ is the Zariski closure of
$\Gamma(f|_{X^0(\CB)})$ and ${G^* \subset X(\CB) \times \CB}$ is the
closure of $\Gamma(f|_{X^0(\CB)})$, then $G^* = (X(\CB) \times \CB) \cap
G$. On the other hand, $G^* = \Gamma(f)$ since $X^0(\CB) \subset
X(\CB)$ is a dense subset and $f$ is continuous. Consequently, $G =
\Gamma$ and ${\Gamma(f) = (X(\CB) \times \CB) \cap \Gamma}$. Thus,
(\ref{prop-9-7-a}) implies (\ref{prop-9-7-b}).

It is clear that (\ref{prop-9-7-b}) implies (\ref{prop-9-7-a}).

For any algebraic subvariety $Z \subset X$, we have
\begin{equation*}
\Gamma(f|_{Z(\CB)}) = \Gamma(f) \cap (Z(\CB) \times \CB).
\end{equation*}
Therefore $f|_{Z(\CB)}$ is a rational function, provided that
(\ref{prop-9-7-a}) and (\ref{prop-9-7-b}) hold.
\end{proof}

\begin{acknowledgements}
We thank J.~Bochnak, C.~Fefferman and J.~Siciak for useful comments, and S.~Yakovenko
for making us aware of the relevance of \cite{bib8} for our project.
Partial financial support to JK   was provided  by  the NSF under grant
number DMS-1362960.
For WK, research was partially
supported by the National Science Centre (Poland) under grant number
2014/15/B/ST1/00046. Furthermore, WK acknowledges with gratitude support
and hospitality of the Max--Planck--Institut f\"ur Mathematik in Bonn.
Partial support for KK was provided by the ANR project STAAVF (France).
\end{acknowledgements}

\phantomsection

\end{document}